\documentclass[12pt,twoside,a4paper]{amsart}
\usepackage[dvipsnames]{xcolor}                     
\usepackage{hyperref} 
\hypersetup{breaklinks, colorlinks,                 
    linkcolor = {Blue},                             
    citecolor = {Blue},                             
    urlcolor  = {Blue}                              %
}                                                   %

\usepackage[T1]{fontenc}
\usepackage[latin1]{inputenc}
\usepackage[left=25mm,top=30mm,bottom=30mm]{geometry}
\usepackage{amssymb}
\usepackage{mathtools}
\usepackage{lmodern}
\usepackage{relsize}   
\usepackage[
final,
scrtime
]{prelim2e}

\renewcommand{\PrelimText}{\footnotesize[\,Version: \texttt{\jobname.tex}\hfill \today\ at \thistime\,]}

      \theoremstyle{plain}
      \newtheorem{Thm}{Theorem}[section]
      \newtheorem{lem}[Thm]{Lemma}
      \newtheorem{Lem}[Thm]{Lemma}

      \theoremstyle{definition}
      \newtheorem{defi}[Thm]{Definition}

      \theoremstyle{remark}
      \newtheorem{Rem}[Thm]{Remark}

   \theoremstyle{definition}\newtheorem{Example}[Thm]{Example}

\renewcommand{\[}{\begin{eqnarray*}}\renewcommand{\]}{\end{eqnarray*}}
\newcommand{\la}{\begin{eqnarray}}\newcommand{\al}{\end{eqnarray}}
\let\oldsqrt\sqrt
\def\sqrt{\mathpalette\DHLhksqrt}
\def\DHLhksqrt#1#2{%
\setbox0=\hbox{$#1\oldsqrt{#2\,}$}\dimen0=\ht0
\advance\dimen0-0.2\ht0
\setbox2=\hbox{\vrule height\ht0 depth -\dimen0}%
{\box0\lower0.4pt\box2}}
\newcommand{\sign}{\mathrm{sgn}}
\newcommand{\sgn}{\mathrm{sgn}} 
\newcommand{\abs}[1]{\left|#1\right|}
\newcommand{\C}{{\mathbb C}}     %
\newcommand{\N}{{\mathbb N}}     %
\newcommand{\R}{{\mathbb R}}     %
\newcommand{\Z}{{\mathbb Z}}%
\renewcommand{\epsilon}{\varepsilon}
\newcommand{\eps}{\varepsilon}
\renewcommand{\rho}{\varrho}

\newcommand{\E}{{\mathbb E}}

\newcommand{\dd}{{\,\mathrm{d}}}
\newcommand{\Prob}{\mbox{\rm Prob}}

\DeclareMathOperator*{\bigconv}{\mbox{\LARGE$\ast$}}

\newcommand{\cC}{{\mathcal C}}
\newcommand{\cF}{{\mathcal F}}
\newcommand{\cP}{{\mathcal P}}

\begin{document}

\author[Mattner]{Lutz Mattner}
\address{Universit\"at Trier, Fachbereich IV -- Mathematik, 54286~Trier, Germany}
\email{mattner@uni-trier.de}

\author[Shevtsova]{Irina Shevtsova}
\address{Lomonosov Moscow State University,
Faculty of Computational Mathematics and Cybernetics,
and 
Russian Academy of Sciences, 
Institute for Informatics Problems of the Federal Research Scientific Center ``Computer Science and Control'', 
Moscow, Russia
}

\email{ishevtsova@cs.msu.ru}

\title[Berry--Esseen for integrals of smooth functions]{An optimal Berry--Esseen type theorem \\ 
for integrals of smooth functions}

\begin{abstract}
We prove a Berry--Esseen type inequality for approximating expectations of sufficiently smooth functions $f$, 
like $f=|\cdot|^3$, with respect to standardized convolutions of laws $P_1,\ldots, P_n$ on the real line
by corresponding expectations based on symmetric two-point laws $Q_1,\ldots,Q_n$ isoscedastic to the $P_i$.
Equality is attained for every possible constellation of the Lipschitz constant $\|f''\|^{}_{\mathrm{L}}$
and the variances and the third centred absolute moments of the $P_i$. The error bound is strictly  smaller than 
$\frac 16$~times the Lyapunov ratio times $\|f''\|^{}_{\mathrm{L}}$, and tends to zero also if $n$ is fixed and the
third standardized  absolute moments of the $P_i$ tend to one.

In the homoscedastic case of equal variances of the $P_i$, and hence in particular in the
i.i.d.~case, the approximating law is a standardized symmetric binomial one.

The  inequality is strong enough to yield for some constellations, in particular in the i.i.d.~case
with $n$ large enough given  the standardized third absolute moment of $P_1$,
an improvement of a more classical and already optimal
Berry--Esseen type inequality of Tyurin~(2009).

Auxiliary results presented include  some inequalities either purely analytical or
concerning Zolotarev's $\zeta$-metrics, and some binomial moment calculations.
\end{abstract}



\thanks{The work was partially supported by the  Russian Foundation for Basic Research (projects 15-07-02984-a and 16-31-60110-mol\_a\_dk) and by the Ministry for Education and Science of Russia (grant No.~MD-2116.2017.1).}


\maketitle
\enlargethispage{1.7\baselineskip}  
\tableofcontents

\section{Introduction and main results}
\subsection{Introduction}                           \label{Subs:Intro}

In statistics and various other applications of probability theory, inconvenient or even intractable distributions are often approximated by relying on some limit theorem. The most popular among such approximations is the \textit{normal} approximation to distributions of sums of a large number $n$ of independent or weakly dependent random variables with appropriate mean and 
variance, which is based on the central limit theorem. However, to use effectively any approximation in practice, one needs an explicit and convenient estimate of its accuracy, and such an estimate may be not as sharp as one might wish. For the purpose of improving the error-bounds one can introduce further terms into the approximating law (leading to the so-called asymptotic expansions) and reach arbitrarily high accuracy, but this requires some additional assumptions on the original distribution. For example, in the case of approximating distributions of sums of independent random variables these conditions are: (i) finiteness of the higher-order moments of the random summands and (ii) some kind of smoothness either of the distributions of the random summands or of the metric under consideration.
 
On the other hand, from the general theory of summation of independent random variables it follows that approximation by infinitely divisible distributions may be more effective even without any moment conditions due to the better error-bound, which is, in the i.i.d.~case and for the Kolmogorov metric, of the order $O(n^{-2/3})$~\cite{Arak1981,Arak1982,Arak_Zaitsev}, rather than $O(n^{-1/2})$ as usual in the CLT, but such an approximation may be  inconvenient, because the sequence of \emph{penultimate} approximating infinitely divisible distributions that guarantees the rate $O(n^{-2/3})$ may be very complicated and usually is not given in an explicit form. Let us recall that an approximation depending on the sample size $n$ not only through location-scale parameters and, in the present context, usually being merely asymptotically normal itself, is sometimes called a \emph{penultimate} approximation, a terminology apparently first introduced in extreme value theory~\cite{Fisher_Tippett}. A recent example of an explicit and convenient penultimate approximation even in the total variation metric, but only for distributions with an absolutely continuous part and finite fourth-order moments, can be found  in~\cite{Boutsikas2015}, where an infinitely divisible shifted-gamma approximation with matching first three moments was proved to have the rate $O(n^{-1})$.

In this paper, as an alternative to the normal approximation, we propose and evaluate another penultimate approximation only assuming finiteness of the third-order moments. Our approximation is in the i.i.d.~case of the same rate $O(n^{-1/2})$ as 
the normal approximation, 
but its error bound depends more favourably on the standardized third absolute moments of the convolved distributions, 
and can in fact tend to zero even for $n$ fixed. 
As the approximating distribution we take the $n$-fold convolution of the symmetric 
two-point laws with the same variances as the original laws, which is asymptotically normal itself.
Thus, in a terminology used for example in~\cite[chapter 4]{LedouxTalagrand1991},  
our approximations are laws of Rademacher averages rather than Gaussian laws.

As a corollary, for the approximation of a standardized characteristic function by  its Taylor polynomial of 
degree~$2$, a new explicit and asymptotically exact error-bound given the absolute third-order moment is
obtained in~\eqref{Eq:IneqCh.F.Taylor(n=1)} below.

Moreover, trivially using the triangle inequality together with the asymptotic normality of the penultimate distribution, which is
valid to a higher order due to vanishing third cumulants and due to 
the smoothness of the metric under consideration, we obtain a sharp upper bound for the accuracy of the normal approximation which improves an already optimal estimate due to Tyurin~\cite{Tyurin2009arxiv,Tyurin2009DAN,Tyurin2010TVP} for some constellations 
(see Theorems~\ref{Thm:Main},~\ref{Thm:Main(noniid)} below). This improvement is possible due to a more favourable dependence of our estimate 
on the moments of the convolved distributions. 

First attempts at a more effective use of the information on the first three moments of the convolved distributions in the estimates of the accuracy of the normal approximation for the Kolmogorov metric were undertaken by Ikeda~\cite{Ikeda1959} and Zahl~\cite{Zahl1963},  followed by Prawitz~\cite{Prawitz1975} and Bentkus~\cite{Bentkus1994}  
(for a detailed review see~\cite[Sections~2.1.1 and~2.4]{Shevtsova2016}). The problem of  optimal use  of moment-type information in the estimates of the accuracy of the normal approximation was posed 
in~\cite{Shevtsova2012DAN1,Shevtsova2012DAN2, Shevtsova2012Debrecen,Shevtsova2012TVP} 
where it was called the problem of \emph{optimization of the structure} of convergence rate estimates and where this problem was partially solved for estimates of the Kolmogorov and the weighted uniform metrics.

To be more precise, we should introduce some notation. Let $\Prob(\R)$ stand for the set of all probability distributions on the real line, $\Prob_s(\R) \coloneqq \{P\in \Prob(\R) : \nu_s(P)\coloneqq \int|x|^s\dd P(x)<\infty\}$ for $s>0,$  $\sigma^2(P) \coloneqq \inf\{ \int(x-a)^2\dd P(x):a\in\R\}$ for $P\in\Prob(\R),$ $\cP_3 \coloneqq \{P\in\Prob_3(\R): \sigma(P)>0\}$, $\mu_k(P) \coloneqq \int x^k\dd P(x)$ for $P\in\Prob_k(\R)$ with $k\in\N$, $\mu(\cdot)\coloneqq \mu_1(\cdot)$. 
We write $\mathrm{N}_{\sigma}$ for the centred normal law on $\R$ with standard deviation
$\sigma\in[0,\infty[$, and $\mathrm{N}\coloneqq\mathrm{N}_{1}$ for the standard normal law  with
distribution function $\Phi$. The one-point law concentrated at $a\in\R$ is denoted by $\delta_a$. 
If $P\in\cP_3$, then we let $\widetilde{P}$ denote its standardization, that is, the image of $P$ under the map $x\mapsto (x-\mu(P))/\sigma(P)$, and
\begin{eqnarray*}
 \varrho(P) &:=&\nu_3\left(\widetilde{P} \right)
 \,\ = \,\ \int \left| \frac{x-\mu(P)}{\sigma(P)}\ \right|^3
 \dd P(x)  \,\ =\,\  \widetilde{P}|\cdot|^3   \,\ =\,\  P \left|\frac{\cdot -\mu(P) }{\sigma(P)} \right|^3     
\end{eqnarray*}
its standardized third absolute moment; of course then $ \varrho(P) \ge 1$, and $\varrho(P)=1$ iff
$\widetilde{P}=\frac12\left(\delta_{-1}+\delta_1 \right)$.  Further, 
let $\widetilde{\cP_3}\coloneqq\{P\in\cP_3\colon \mu(P)=0,\ \sigma(P)=1\}=\{\widetilde P\colon P\in\cP_3\}$.
The tilde notation just introduced should not lead to confusion with a more standard one, used also here, 
for indicating equality of laws of random variables, as in $X\sim Y$, or 
for specifying the law of a random variable, as in $X\sim P$. For $P,P_1,\ldots,P_n\in\Prob(\R)$ let further $\bigconv_{i=1}^n P_i$ denote the convolution of the laws $P_1,\ldots,P_n$,  and   $P^{\ast n}$  the $n$th convolution power of $P$.

With the above notation, the problem of optimization of the structure of \textit{asymptotic} convergence rate estimates stated in~\cite{Shevtsova2012DAN1,Shevtsova2012DAN2, Shevtsova2012Debrecen,Shevtsova2012TVP} may be formulated as follows: Find the pointwise greatest lower bound to all functions $g:\mathopen[1,\infty\mathclose[\to\R_+$ such that we have
\begin{equation}\label{Eq:UniMetricOptimalStructProblem}
\Delta_n(P) \ \coloneqq  \ 
\sup_{x\in\R}\abs{\widetilde{P^{\ast n}} \big(\mathopen]-\infty,x]\big)-\Phi(x)}\ \le\  \frac{g(\rho(P))}{\sqrt{n}}+\eps_n(P)
\quad \text{ for $P\in\cP_3$ and $n\in\N$} 
\end{equation}
for some remainder term $\eps_n(P) \ge 0$, possibly depending on $g$,  satisfying
\begin{equation}\label{Eq:CLTRemTermForL->0}
 \lim_{\ell\to0}\frac{\eps(\ell)}{\ell}  \ \ =\ \ 0 \quad \text{with}\quad \eps(\ell)\ \coloneqq\ \sup_{P\in\cP_3,n\in\N\colon \rho(P) = \ell\sqrt{n}}\eps_n(P) \quad\text{for}\quad \ell>0,
\end{equation}
that is, $\eps_n(P)=o(\varrho(P)/\sqrt{n})$ for $\varrho(P)/\sqrt{n} \rightarrow 0$ with not only $n$ but also $\varrho(P)$ allowed to vary. It is easy to see that for any $g$ satisfying~\eqref{Eq:UniMetricOptimalStructProblem} and~\eqref{Eq:CLTRemTermForL->0} with  some $\eps_n(P)$ we have  $g\ge g_\ast$, where
$$
g_*(\rho)\,\ \coloneqq\,\  \lim_{\ell\to0}
\sup \left\{\sqrt{n}\Delta_n(P)\colon n\in\N,\ P\in\cP_3,\
   \rho(P)=\rho \le \ell\sqrt{n}  \right\}  \quad\text{ for } \rho\in\mathopen[1,\infty\mathclose[\, ;
$$
moreover, for $g=g_*$ and, say, $\eps_n(P)\coloneqq \max\big\{0,\Delta_n(P) - g(\varrho(P))/\sqrt{n}\big\}$,
we have~\eqref{Eq:UniMetricOptimalStructProblem} and
\begin{equation}\label{Eq:CLTRemTermForL->0NonuniForRho>=1}
\limsup_{\ell\to0}\sup_{n\in\N,P\in\cP_3\colon \rho(P)=\rho=\ell\sqrt{n}}\sqrt{n}\eps_n(P) \ \ =\ \ 0  
   \quad\text{ for } \rho\in\mathopen[1,\infty\mathclose[,
\end{equation}
which is weaker than~\eqref{Eq:CLTRemTermForL->0}
since $\varrho(P)$ is fixed in the supremum in \eqref{Eq:CLTRemTermForL->0NonuniForRho>=1}.
However, in some cases, it is possible to construct $\eps_n(P)$ satisfying the stronger condition~\eqref{Eq:CLTRemTermForL->0} such that inequality~\eqref{Eq:UniMetricOptimalStructProblem} holds with $g=g_*$. In what follows, we call $g_*$ the \textit{optimal} function.

The problem of explicitly determining  the optimal function $g_*$ is very complicated. Historically the first investigations were done for  analogous problems with either the functions $\epsilon_n\ge 0$  in~\eqref{Eq:UniMetricOptimalStructProblem} 
only required to satisfy a version of~\eqref{Eq:CLTRemTermForL->0} pointwise rather than uniformly in $P$, namely
\begin{equation}\label{Eq:CLTRemTermForL->0NonuniForPfromP3}
  \sup_{P\in\cP_3}\lim_{n\to\infty}\sqrt{n}\eps_n(P)\,\ =\,\ 0 
\end{equation}
(which is even weaker than~\eqref{Eq:CLTRemTermForL->0NonuniForRho>=1}), 
solved by Esseen~\cite{Esseen1945,Esseen1956} and thus yielding a lower bound for~$g_*$, or the functions~$g$ satisfying~\eqref{Eq:UniMetricOptimalStructProblem} restricted to be linear (without constant term), where Chistyakov~\cite{Chistyakov2001}, significantly sharpening the work of Esseen~\cite{Esseen1956}, eventually found the optimal one. More precisely, restricting now attention to $P\in\widetilde{\cP_3}$ rather than  ${P\in\cP_3}$
for notational convenience and without loss  of generality,
from Esseen's~\cite{Esseen1945}  short Edgeworth expansion
$$
\widetilde{P^{\ast n}} \big(\mathopen]-\infty,x]\big) \,\ =\,\ \Phi(x)+ (1-x^2)\mathrm{e}^{-x^2/2}\cdot \frac{\mu_3(P)}{6\sqrt{2\pi n}}+
 \psi_n(x)\mathrm{e}^{-x^2/2}\cdot \frac{h(P)}{\sqrt{2\pi n}} + o\left(\frac1{\sqrt{n}}\right)
$$
valid as $n\to\infty$ uniformly in $x\in\R$ for every fixed $P\in\widetilde{\cP_3}$, 
where $h(P)$ is the span in case of a lattice distribution~$P$ and $h(P)=0$ otherwise, and $\psi_n$ is a certain $(h(P)/\sqrt{n})$-periodic $[-\frac12,\frac12]$-valued function, in~\cite{Esseen1956} Esseen, first, deduced that
\begin{equation}\label{Eq:EsseenLimDelta_nSqrt(n)}
\lim_{n\to\infty}\sqrt{n}\Delta_n(P)\,\ =\,\  \frac{|\mu_3(P)|+3h(P)} {6\sqrt{2\pi}} \quad\text{for }\ P\in\widetilde{\cP_3}.
\end{equation}
Second, he considered and solved an extremal problem yielding an exact upper bound of the 
R.H.S.~of~\eqref{Eq:EsseenLimDelta_nSqrt(n)} in terms of $\rho(P)$ only, namely, he proved that
\la                                             \label{Eq:Esseen_extremal_result}
\sup_{P\in\widetilde{\cP_3}} \frac{\mu_3(P)+3h(P)} {\rho(P)} & =&  \sqrt{10}+3,
\al
with equality attained iff  $P=P_{\rho^{}_\mathrm{E}}$, where for $\rho\in[1,\infty[$ here and below, $P_\rho\in\widetilde{\cP_3}$ denotes the  two-point distribution uniquely defined by the conditions $\mu_3(P_\rho)\ge0$ and $\nu_3(P_\rho)=\rho,$ namely
\la                                                                                       \label{Eq:Def_P_rho}
 &&P_\rho\left(\left\{-\sqrt{\tfrac pq}\,\right\}\right)=q\coloneqq1-p,\quad  P_\rho\left(\left\{\sqrt{\tfrac qp}\,\right\}\right)=  p=p_\rho\coloneqq\frac12-\frac12\sqrt{\frac{\rho}2\sqrt{\rho^2+8}-\frac{\rho^2}2-1},
\al
and having the span and the third moment 
\la
 h_\rho &\coloneqq& h(P_\rho)\,\ =\,\ 1/\sqrt{pq} \,\ =\,\ 2\sqrt2\Big/\sqrt{ \rho^2-\rho\sqrt{\rho^2+8}+4}\,,
           \nonumber                \\
 B(\rho) & \coloneqq& \mu_3(P_\rho) \,\ =\,\ (q-p)/\sqrt{pq} \,\ =\,\ \sqrt{\rho^2/2 +\rho\sqrt{\rho^2+8}/2 -2}\,,  
  \label{Eq:Def_B(rho)_new}        
\al
and where 
$$
   \rho^{}_\mathrm{E} \,\ :=\,\ \sqrt{20(\sqrt{10}-3)/3}\,\ =\,\ 1.0401\ldots
$$
corresponds  to $p^{}_E\coloneqq p_{\rho^{}_\mathrm{E}}   =   (4-\sqrt{10})/2=0.4188\ldots\,.$ 
Hence,~\eqref{Eq:EsseenLimDelta_nSqrt(n)} specialized to $P=P_\rho$ yields the lower bound
$$
g_*(\rho)\,\ \ge\,\  \frac{B(\rho)+3h_\rho}{6\sqrt{2\pi}} \,\ =\,\ \frac1{6\sqrt{2\pi}}\cdot \frac{2\sqrt{\rho\sqrt{\rho^2+8}-\rho^2-2}+6\sqrt2} {\sqrt{\rho^2-\rho\sqrt{\rho^2+8}+4}}\,\  \eqqcolon\,\   g_0(\rho),\quad \rho\ge1,
$$
while combination of~\eqref{Eq:EsseenLimDelta_nSqrt(n)} and~\eqref{Eq:Esseen_extremal_result} allowed  Esseen to find the so-called \textit{asymptotically  best} constant 
\la                                                                             \label{Eq:Esseen_(1956)_constant}
  C_\mathrm{E} &\coloneqq& \sup_{P\in\cP_3} \lim_{n\to\infty} \frac{\sqrt{n}\Delta_n(P)}{\rho(P)} 
 \,\ = \,\ \frac{\sqrt{10}+3}{6\sqrt{2\pi}}\,\ =\,\ 0.4097\ldots
\al
and consequently the necessary condition $c\ge C_\mathrm{E}$ for~\eqref{Eq:UniMetricOptimalStructProblem} to hold with the linear function ${g(\rho)=c\rho}$. Let us remark in passing that~\cite{DinevMattner} presents 
$\sqrt{2/\pi}= 0.7978\ldots$ as the asymptotically best constant  analogous to $C_\mathrm{E}$ when arbitrary intervals replace the unbounded  ones $\mathopen]-\infty,x\mathclose]$ in the definition of $\Delta_n(P)$ in~\eqref{Eq:UniMetricOptimalStructProblem}. 

We observe that~\eqref{Eq:UniMetricOptimalStructProblem} holds with $g=g_0$ and $\eps_n(P)$ satisfying the weakest condition \eqref{Eq:CLTRemTermForL->0NonuniForPfromP3}. About 40 years after Esseen's work~\cite{Esseen1956}, 
Chistyakov~\cite{Chistyakov1996,Chistyakov2001}  finally managed to find in particular the value of the \textit{asymptotically exact} constant
\la                    \label{Eq:Chistyakov_constant}
\lim_{\ell\to0}\sup\left\{ \frac{\sqrt{n}\Delta_n(P)}{\rho(P)}\colon n\in\N,\ P\in\cP_3,\ \rho(P)\le\ell\sqrt{n} \right\}
&=& C_\mathrm{E} 
\al
(we are not aware of any really convincing names for the ``Esseen constant'' in~\eqref{Eq:Esseen_(1956)_constant}
and the ``Chistyakov constant'' in~\eqref{Eq:Chistyakov_constant}; the ones used above are at least compatible with some earlier literature such 
as~\cite{Shevtsova2010TVP})
and to prove that~\eqref{Eq:UniMetricOptimalStructProblem} and~\eqref{Eq:CLTRemTermForL->0} hold true with
$$
g(\rho)\,\ =\,\   C_E\cdot\rho  \,\ \eqqcolon \,\  g_1(\rho) \quad \text{ for }\rho\ge1
$$
and $\eps(\ell)=O\left(\ell^{40/39}|\log\ell|^{7/6}\right)$ (see~\cite{Chistyakov2001}). We also remark here that 
the papers~\cite{Chistyakov1996,Chistyakov2001} treat the non-i.i.d.~case and that the results cited above represent the corresponding specializations to the i.i.d.~case. In the more recent paper~\cite[Corollary~4.18 on p.\,303]{Shevtsova2012Debrecen} Chistyakov's upper bound for $\eps(\ell)$ was improved to $\eps(\ell)\le4\ell^{4/3}$ in the general case and $\eps(\ell)\le3\ell^2$ in the i.i.d. case.

Discarding now the restriction to linear functions, we note that Chistyakov's result reported above 
yields $g_*(\rho)\le g_1(\rho)$ for every $\rho\in[1,\infty[$, with equality in case of ${\rho =\rho_\mathrm{E}}$,
and that a result of Hipp and Mattner~\cite{HippMattner2007} yields $g_*(1) =1/\sqrt{2\pi}=0.3989\ldots
< g_1(1)$. The recent papers~\cite{Shevtsova2012DAN1}, \cite[Theorem\,4.13, Corollary\,4.17]{Shevtsova2012Debrecen} succeeded
in particular in proving that in fact the equality $g_* = g_0$ holds on an interval containing 
the previously treated  points $1$ and $\rho_\mathrm{E}$, namely,
\[
g_*(\rho) &=& g_0(\rho)\quad \text{for } 1\ \le\ \rho\ \le\ \rho_0\ \coloneqq\  3^{1/4}(4-\sqrt3)/\sqrt{6} \,\ =\,\ 1.2185\ldots,
  \\ 
g_*(\rho) &\le& g_2(\rho)\quad \text{for all }\rho\ge1
\]
with
\[
g_2(\rho) &\coloneqq& \frac{2\rho}{3\sqrt{2\pi}}+\sqrt{\frac{2\sqrt3-3}{6\pi}} 
\quad \text{ for }\rho\ge1,
\]
$g_2(\rho_0)=g_0(\rho_0),$ and $g_2$ is asymptotically optimal for $\rho\to\infty$ in the sense of
$$
\lim_{\rho\to\infty}\frac{g_*(\rho)}{g_2(\rho)} \,\ =\,\ 1.
$$
Observe that $g_2(\rho)<C_E\cdot\rho$ for  $\rho>\frac23\sqrt{2/\sqrt3-1}(\sqrt{10}+1)=1.0914\ldots,$ in particular, for $\rho\ge\rho_0$, and that each of the functions $g_1$ and $g_2$ is tangent to $g_*$ at the points $\rho^{}_{\mathrm E}$
and $\rho_0$, respectively. We also note that, as it follows from Schulz' thesis~\cite[p.~16]{Schulz2016}, the equality $g_*=g_0$ cannot hold on the whole ray $[1,\infty[$ even in the binomial case, namely,
\[
g_*(\rho) &>& g_0(\rho) \quad\text{for}\quad \rho>3.8021\ldots,
\]
which corresponds to $p_\rho<p_*,$ where $p_*=0.05822\ldots$ is the unique root of the equation $7-130p+165p^2+50p^3-23p^4=0$ on the interval $0<p<1/3.$

In~\cite{Shevtsova2012DAN1}, \cite[Theorem~4.13 on p.\,298 and Corollary~4.17 on p.\,302]{Shevtsova2012Debrecen} there have also been obtained explicit uniform upper bounds for the remainder term 
$\eps_n(P)$ in~\eqref{Eq:UniMetricOptimalStructProblem} and~\eqref{Eq:CLTRemTermForL->0} with the continuous function $g$ defined by $g(\rho)\coloneqq g_0(\rho)$ for $\rho\in[1,\rho_0]$ and $g(\rho)\coloneqq g_2(\rho)$ for $\rho>\rho_0$, namely, $\eps(\ell)\le2\ell^{3/2}$ for all $\ell>0$.  

Moreover, in the same papers an extension of~\eqref{Eq:UniMetricOptimalStructProblem} to the non-i.i.d. case was obtained in the form
$$
\sup_{n\in\N,\,P_1,\ldots,P_n\in\cP_3}\sup_{x\in\R}
\bigg| \widetilde{\bigconv_{i=1}^nP_i} \big(\mathopen]-\infty,x]\big)-\Phi(x)\bigg|
\,\ \le\,\  \tau\cdot g(\ell/\tau)+3\ell^{7/6}
$$
with the same function $g$ as defined in the preceding paragraph for the case of coinciding $P_1,\ldots,P_n$,  where the supremum is taken over all $n$ and all centred distributions $P_1,\ldots,P_n\in\cP_3$ such that $\sum_{i=1}^n\nu_3(P_i)/(\sum_{i=1}^n\sigma_i^2)^{3/2}=\ell$,  $\sum_{i=1}^n\sigma_i^3/(\sum_{i=1}^n\sigma_i^2)^{3/2}=\tau$, $\sigma_i^2\coloneqq\sigma^2(P_i)$ for $i\in\{1,\ldots,n\}$. Moreover, there has also been proved a sharpened upper bound for $\eps_n(P)$ in~\eqref{Eq:UniMetricOptimalStructProblem} and~\eqref{Eq:CLTRemTermForL->0}  with the linear function $g(\rho)=C_\mathrm{E}\cdot\rho$, namely, $\eps(\ell)\le3\ell^2$ in the i.i.d.~case and $\eps(\ell)\le4\ell^{4/3}$ in the non-i.i.d.~case. These bounds improve the earlier results of Bentkus~\cite{Bentkus1994} and Prawitz~\cite{Prawitz1975}.

Recently Schulz~\cite[Theorem~1 on p.~1]{Schulz2016} proved that the remainder term $\eps_n(P)$ in~\eqref{Eq:UniMetricOptimalStructProblem} with $g=g_0$ can be omitted in case of two-point distribution $P=P_\rho$ for $1\leq\rho\leq5\sqrt{2}/6=1.1785\ldots$ (which corresponds to $p_\rho\in[1/3,1/2]$),  generalizing the earlier result by Hipp and Mattner~\cite{HippMattner2007} originally obtained for $\rho=1$. Also, in~\cite[Theorem~1 on p.~1]{Schulz2016}  it is proved that $\Delta_n(P_{\rho})\le C_\mathrm{E}\cdot \rho$ for every $\rho\ge1$ and $n\in\N$.

The present paper can in its main parts be regarded as a transfer and then improvement of some of the above results 
from the Kolmogorov to the appropriate Zolotarev metric, namely $\zeta_3$. 

For the related topic of asymptotic expansions of expectations of smooth functions in the CLT, where rigorous results 
go back at least to Cram\'er~\cite[p.~45, (41a)]{Cramer_1928} in the case of characteristic functions, 
and to von~Bahr~\cite{von_Bahr_1965} in the case of moments and absolute moments,
we may refer in chronological order to the surveys in~\cite[section 25, that section 
apparently unchanged from its earlier 1986 edition]{Bhat_Ranga_Rao}, \cite[Chapter~2]{Ghosh1994},
and \cite[pp.~196--197]{Petrov_1995},
and to the more recent papers \cite{BorisovSkilyagina1996,BorisovPanchenkoSkilyagina1998,Jiao2012}.
From the vast literature on asymptotic expansions of distribution functions, and thus expectations of 
certain non-smooth functions, for which  one may also  consult the monographs just cited,  
let us mention only the recent paper~\cite{AngstPoly2017}.

This paper is organized as follows. Subsections~\ref{Subsec:Notation},
\ref{Subsec:Main_results}, and~\ref{Subsec:NormApprox} present exact formulations of the
main results with discussion. Sections~\ref{Sec:Main_proofs}  and~\ref{Sec:zeta3-between-N-B} contain the proofs of the main results. The latter are based on Hoeffding's~\cite{Hoeffding1955} and Tyurin's~\cite{Tyurin2009arxiv,Tyurin2009DAN,Tyurin2010TVP} results for extremal values of linear and quasi-convex functionals under given moment conditions treated in a novel way in section~\ref{Sec:reduction-thms}, 
the previously obtained bound on the third-order moment given the absolute third-order moment~\cite{Shevtsova2014JMAA} as well as a new exact absolute third moment recentering inequality presented  in Lemma~\ref{Lem:E|X-t|^3<=E(a+bX+cX^2+d|X|^3)}, various properties of $\zeta$-metrics, in particular  in connection with the $s$-convex ordering~\cite{DenuitLefevreShaked1998} as treated in section~\ref{Sec:zeta-metrics}, and the properties of the Krawtchouk polynomials~\cite{MacWilliams_Sloane} associated to the symmetric binomial law used in section~\ref{Sec:zeta3-between-N-B}.

The main results of this paper have been announced without proofs in \cite{Mattner_Shevtsova_1}.

\subsection{Further notation, properties of the function $B$}               \label{Subsec:Notation}
 
Terms like ``positive'', ``increasing'', and  ``convex''  are understood in the wide sense, 
adding ``strictly'' when  appropriate.  Also, ``interval'' may refer to  any convex subset of $\R$, possibly degenerated to one point 
or even to the empty set. 
We use the de Finetti indicator notation, $(\text{statement})\coloneqq 1$ if ``statement'' is true,
$(\text{statement})\coloneqq 0$ otherwise, for example in~\eqref{Eq:Ineq_S(f)_S(f')} below.

If $I\subseteq\R$ is an interval and $E$ is a Banach space over $\R$ or $\C$, and with its norm denoted by $| \cdot |$ 
since the most interesting cases here are $E=\R$ and $E=\C$, then we use the standard notation  $\cC(I,E)$ for the continuous $E$-valued functions on $I$, $\cC^m(I,E)$ for the ones $m\in\N_0$ times continuously differentiable, and $\cC^{m,\alpha}(I,E)$  for those $f\in\cC^m(I,E)$ whose $m$-th derivative $f^{(m)}$ has a finite H\"older  constant 
\la                                                    \label{Eq:Def_Hoelder-constant}
   \|f^{(m)}\|^{}_{\mathrm{L},\alpha} &\coloneqq&  \sup_{x,\,y\in I,\,x\neq y} \frac{|f^{(m)}(x)-f^{(m)}(y)|}{|x-y|^\alpha}
\al
of order  $\alpha\in\mathopen]0,1\mathclose]$.
It is well known that for $E$ finite-dimensional, and also more generally as discussed in~\cite{DiestelUhl}, the condition 
$f\in\cC^{m,1}(I,E)$ is equivalent to $f^{(m)}$ being absolutely continuous with its then Lebesgue-almost everywhere existing derivative 
$f^{(m+1)}$ satisfying  
\[
 \|f^{(m)}\|^{}_{\mathrm{L}} & \coloneqq & \|f^{(m)}\|^{}_{\mathrm{L},1}\ =\ \|f^{(m+1)}\|^{}_{\infty}\\
& \coloneqq & \inf \{ M \in \R : |f^{m+1}| \le M \text{ Lebesgue-almost everywhere on } I\}.
\]

Recalling the definition of $B(\rho)$ given in~\eqref{Eq:Def_B(rho)_new} for $\rho\in[1,\infty[$, let us also put
\la &&                                \label{Eq:Def_A_B}
  A(\rho) \,\ \coloneqq\,\ \rho^{-1}B(\rho)\,\ =\,\ \sqrt{\tfrac12\sqrt{1+8\rho^{-2}}+\tfrac12 -2\rho^{-2} }
 \quad\text{ for }\rho\in[1,\infty[.
\al
The notation $A$ here is as used in~\cite[pp.~194, 208]{Shevtsova2014JMAA},
so let us note that there is an inconsequential typo in the formula for $A'(\rho)$ in
\cite[p.~208]{Shevtsova2014JMAA}, where $\rho^{3/2}$  should be $\rho^3/2$.

\begin{Lem}                                              \label{Lem:A_and_B}
The functions $A$ and $B$ are continuous, strictly concave and increasing,
with $A(1)=  B(1)=0,$ $\lim_{\rho\rightarrow1}A(\rho)/\sqrt{\rho-1} = \sqrt{8/3},$
and $\lim_{\rho\rightarrow\infty}A(\rho)=1$.
In particular, we have $0 < A(\varrho) <1$ and 
$\varrho-1 < B(\rho)<\rho$ for  $\varrho\in\mathopen]1,\infty[$.
\end{Lem}
\begin{proof} We have  $\left(A^2(\rho)\right)' = 4\rho^{-3}(1-(1+8\rho^{-2})^{-1/2})$
strictly decreasing  and positive  for $\rho\in[1,\infty[$,
hence $A^2$ strictly concave and increasing, thus $A=\sqrt{A^2}$ strictly   concave and increasing  as well,
and also $\lim_{\rho\rightarrow1}A^2(\rho)/(\rho-1)=\left(A^2\right)'(1)=8/3$.
$B$ is obviously strictly increasing and, by~\cite[p.~209]{Shevtsova2014JMAA}, satisfies $B''<0$ and is hence
strictly concave; hence $B(\varrho)/(\varrho-1) = (B(\varrho)-B(1)) /(\varrho-1)$ is strictly decreasing 
and hence $>1$.
\end{proof}

For $n\in\N=\{1,2,\ldots\}$, let $\mathrm{B}_{n,\frac12} \coloneqq (\frac12(\delta_{0}+\delta_1))^{\ast n}$
denote the binomial law with 
$\mathrm{B}_{n,\frac12}(\{k\}) =\mathrm{b}_{n,\frac12}(k)\coloneqq \binom{n}{k}2^{-n}$ for $k\in\N_0\coloneqq \N\cup\{0\}$.

\subsection{The main result (Rademacher average approximation) and some consequences}   \label{Subsec:Main_results}

Our main result is: 

\begin{Thm}                                               \label{Thm:Main}
Let $n\in \N,$ $P_1,\ldots,P_n\in\cP_3,$  $E$ be a  Banach space, and $f \in \cC^{2,1}(\R,E)$.  Then we have
\la                        \label{Eq:Main_inequality_non-i.i.d.}
  \left|\widetilde{\bigconv_{i=1}^n P_i} \, f  - \widetilde{\bigconv_{i=1}^n Q_i}\, f  \right|
   &\le&   \frac{\|f''\|^{}_{\mathrm{L}}}{6} \sum_{i=1}^n
            \frac{\sigma_i^3 }{\sigma^3}B(\varrho_i) 
\al
with $\sigma_i \coloneqq\sigma(P_i)$, 
$Q_i\coloneqq \frac 12(\delta_{-\sigma_i}+\delta_{\sigma_i}),$ 
$\sigma\coloneqq\left(\sum_{i=1}^n\sigma_i^2\right)^{1/2},$
and $\varrho_i \coloneqq \varrho(P_i)$. 
If each  $P_i$ is  a two-point law and if the centred third moments of the $P_i$ are all $\ge 0$ or all $\le0,$ and if also $f(x)=cx^3$ for $x\in\R,$ with a constant $c\in E,$ then equality holds in~\eqref{Eq:Main_inequality_non-i.i.d.}.
\end{Thm}

The proof of Theorem~\ref{Thm:Main} is given in section~\ref{Sec:Main_proofs} on p.\,\pageref{Sec:Main_proofs}.

Clearly, in the homoscedastic case of $\sigma_1 = \ldots =\sigma_n$, the approximating
law  $\widetilde{\bigconv_{i=1}^n Q_i}$ in Theorem~\ref{Thm:Main} is just the
standardized symmetric binomial law $\widetilde{\mathrm{B}_{n,\frac12}}$.
And in the i.i.d.~case of $P_1=\ldots=P_n\eqqcolon P$,
inequality~\eqref{Eq:Main_inequality_non-i.i.d.} further simplifies to
\la                                                  \label{Eq:Main_inequality}
  \left|\widetilde{P^{\ast n}} f -  \widetilde{\mathrm{B}_{n,\frac12}}f  \right|
   &\le& \frac{B(\rho(P))}{6\sqrt{n}} \|f''\|_{\mathrm{L}}^{},
\al
with equality whenever $P$ is a two-point law and $f(x)=cx^3$.

Here are three examples of applications of Theorem~\ref{Thm:Main}, of which the first one,
however, is a mock one.

\begin{Example}               \label{Example:Thm_6_of_Irina_2014}
Theorem~\ref{Thm:Main} formally yields~\cite[Theorem~6]{Shevtsova2014JMAA}, namely
\la               \label{Eq:Thm_6_of_Irina_2014}
 \max_{P\in \cP_3\colon \rho(P)=\rho} \left|\int x^3\dd\widetilde{P}(x) \right|
  &=& B(\rho) \quad\text{ for }\rho\in[1,\infty[
\al
with equality attained for two-point laws,
by applying~\eqref{Eq:Main_inequality} with $E=\R$, $n=1$, and $f(x) \coloneqq x^3$,
since for $P\in\cP_3$, we have
\[
 \left|\int x^3\dd\widetilde{P}(x) \right|
 &=& \left|\widetilde{P} f - \widetilde{\mathrm{B}_{1,\frac12}} f \right|
\]
and $\|f''\|^{}_{\mathrm{L}} = 6$.
However, \eqref{Eq:Thm_6_of_Irina_2014} 
is used in Step 6 of our proof of Theorem~\ref{Thm:Main}. 
\end{Example}

\begin{Example}                                                \label{Example:Char_function_approximation}
In Theorem~\ref{Thm:Main}, let $E=\C$ and $f(x) = \mathrm{e}^{\mathrm{i}tx}$ for some $t\in\R$.
Then, writing~$\varphi$ 
for the characteristic function of $\widetilde{\bigconv_{i=1}^n P_i}$, we get
\la                                                             \label{Eq:Ineq_char_fct}
  \left| \varphi (t) - \prod_{i=1}^n \cos\left(\frac{\sigma^{}_it}{\sigma}  \right) \right|
  &\le & \frac{|t|^3}{6}\sum\limits_{i=1}^n \frac{\sigma_i^3 B(\varrho_i)}{\sigma^3},
\al
since here $\| f'' \|^{}_{\mathrm{L}} = \sup_{x\in\R} |f'''(x)| = |t|^3$. In~\eqref{Eq:Ineq_char_fct}, we have
asymptotic equality for $t\rightarrow 0$ if  all the $P_i$ are two-point laws with
equi-signed third centred  moments, by equality in~\eqref{Eq:Main_inequality_non-i.i.d.}
for $f=(\cdot)^3$ and by a Taylor expansion inside the modulus on the left hand side of~\eqref{Eq:Ineq_char_fct}.

Moreover, using 
\la                                                             \label{Eq:prod_cos_inequality}
  0&\le& \prod_{i=1}^n \cos t_i -1+ \frac{1}{2}\sum_{i=1}^nt_i^2 
    \,\ \le \,\ \frac1{24}\sum_{i=1}^n t_i^4 + \frac1{4}\sum_{i<j} t_i^2t_j^2  \quad\text{ for } t\in\R^n,
\al
which follows by rewriting the central term in~\eqref{Eq:prod_cos_inequality}
with the help of  i.i.d.~Rademacher variables $\xi_1,\ldots,\xi_n$ as 
\[
  \prod_{i=1}^n \E \mathrm{e}^{\mathrm{i}t_i\xi_i} - 1+ \frac{1}{2}\sum_{i=1}^nt_i^2 \E\xi_i^2
   \,\ =\,\ \E\bigg(\cos\Big(\sum_{i=1}^nt_i\xi_i\Big)-1+ \frac{1}{2}\Big(\sum_{i=1}^nt_i\xi_i\Big)^2\bigg)
\]
and applying $0\le \cos x -1+\frac12x^2 \le \frac1{24}x^4$ inside the last expectation above,
we obtain from~\eqref{Eq:Ineq_char_fct} the following estimate for the accuracy of the 
approximation of $\varphi$ by the first terms of its Taylor expansion:
\[
  \left| \varphi (t) - 1 + \frac{t^2}{2}\right|
 &\le& \frac{|t|^3}{6}\sum\limits_{i=1}^n \frac{\sigma_i^3 B(\varrho_i)}{\sigma^3} +\frac{t^4}{24}\sum_{i=1}^n \frac{\sigma_i^4}{\sigma^4} + \frac{t^4}{4}\sum_{i<j}\frac{\sigma_i^2\sigma_j^2}{\sigma^4}\quad \text{ for } t\in\R.
\]
In particular, with $n=1$  we have 
\la                                                       \label{Eq:IneqCh.F.Taylor(n=1)}
\abs{\E \mathrm{e}^{\mathrm{i}tX}-1+ \frac{t^2}2} &\le& A(\varrho)\frac{\varrho |t|^3}6+\frac{t^4}{24}
\al
for all $t\in\R$ and an arbitrary r.v. $X$ with $\E X=0,$ $\E X^2=1$, $\varrho:=\E|X|^3<\infty$, where the inequality turns into the asymptotic equality as $t\to0$ whenever $X$ is a two-point r.v. (more precisely, either $X~\sim P_\rho$ or $-X\sim P_\rho$ with  $P_\rho$ defined in~\eqref{Eq:Def_P_rho}).

Inequality~\eqref{Eq:IneqCh.F.Taylor(n=1)} for small $t$ improves the bound
\[
 \abs{\E \mathrm{e}^{\mathrm{i}tX}-1+ \frac{t^2}2}&\le& \frac{\varrho|t|^3}6\inf_{0<\lambda<1/2}\{\lambda A(\varrho)+q_3(\lambda)\}
\]
obtained in~\cite[Corollary 4]{Shevtsova2014JMAA}, where 
\[
 q_3(\lambda) &\coloneqq& \sup_{x>0} \frac6{x^3}\abs{\mathrm{e}^{\mathrm{i}x} -1-\mathrm{i}x-\frac{(\mathrm{i}x)^2}2-\lambda\frac{(\mathrm{i}x)^3}6} \,\ \ge \,\ 1-\lambda \quad\text{ for } 0\le\lambda\le1/2,
\]
with the final inequality following from considering $x\downarrow 0$.
Indeed, for every $\rho\ge 1$, we have  $A(\rho)<1$  by Lemma~\ref{Lem:A_and_B}  and hence get  
\[
 \inf_{0<\lambda<1/2}\{\lambda A(\varrho)+q_3(\lambda)\} &\ge& 
 \inf_{0<\lambda<1/2}\{\lambda A(\varrho)+ 1- \lambda\} 
 \,\ =\,\ \frac{A(\rho)+1}2 \,\ > \,\   A(\rho) . 
\]
\end{Example}

\begin{Example}
Applying Theorem~\ref{Thm:Main} to $E=\R$ and $f=|\cdot|^3$ in the i.i.d.~case yields: For i.i.d. $X_i\sim P \in\cP_3$, we have
\la        \label{Eq:Third_abs_moment_ineq}
 \left| \E \Big|\widetilde{\textstyle \sum_{i=1}^nX_i}\Big|^3- \widetilde{\mathrm{B}_{n,\frac12}}
    |\cdot|^3    \right|  &\le& \frac{B(\varrho(P))}{\sqrt{n}},
\al
by $\| f'' \|^{}_{\mathrm{L} } = 6$, where by formula~\eqref{Eq:Cent_3rd_abs_mom_sym_bin} stated and proved below, we have explicitly 
\[
\widetilde{\mathrm{B}_{n,\frac12}}   |\cdot|^3  &=&
 \left\{\begin{array}{ll}  \left(2n^{\frac12} +n^{-\frac12}- n^{-\frac32}\right) \mathrm{b}^{}_{n,\frac12}(\lfloor\frac{n}2\rfloor)
    & \text{ if $n$ is odd},\\
   2n^{\frac12}\mathrm{b}^{}_{n,\frac12}(\frac{n}2)  & \text{ if $n$ is even.} \end{array}\right.
\]
\end{Example}

Let us note that $\widetilde{\mathrm{B}_{n,\frac12}}|\cdot|^3$ can not be replaced by any other function of $n$ without invalidating~\eqref{Eq:Third_abs_moment_ineq}, since the R.H.S. of~\eqref{Eq:Third_abs_moment_ineq} is zero if the $X_i$ are symmetrically Bernoulli-distributed; an analogous remark applies to every application of Theorem~\ref{Thm:Main}  in the i.i.d. case.

In Theorem~\ref{Rem:Main_thm_rewritten} below, we rewrite Theorem~\ref{Thm:Main} in terms of 
Zolotarev's distance $\zeta_3$. On the one hand this actually prepares for the proof of  Theorem~\ref{Thm:Main}.
On the other hand it allows, by simply using the triangle inequality combined with Theorem~\ref{Thm:epsilon_n} 
below, to obtain the quite sharp normal approximation result in Theorem~\ref{Thm:Normal_appr}.
Since in turn the proof of Theorem~\ref{Thm:epsilon_n} uses $\zeta_4$, let us recall here the definition
and  some basic and well-known properties of $\zeta_s$ in general. For more properties of Zolotarev distances 
needed in the present paper, including new results as well as apparently previously unpublished detailed proofs  
of some ``well-known'' results, we refer to section~\ref{Sec:zeta-metrics} below. Standard references on 
$\zeta$-distances include the monographs~\cite[Chapter~1]{Zolotarev1986}, \cite{Rachev1991}, \cite[Chapter~2]{Senatov1998}.

We will use the notation introduced around~\eqref{Eq:Def_Hoelder-constant}, here with $I=E=\R$.

\begin{defi}[$\zeta$-distances]                                              \label{Def:zeta-distances}
Let $s>0$. With $m\coloneqq \lceil s-1\rceil \in \N_0$ and 
$\alpha \coloneqq s-m \in\mathopen]0,1\mathclose]$, we put
\[ 
   \cF_s &\coloneqq&  \{ f \in \cC^{m,\alpha}(\R,\R) : \|f^{(m)}\|^{}_{\mathrm{L},\alpha} \le 1\}, \qquad
  \cF_s^\infty \,\ \coloneqq \,\  \{   f \in \cF_s : f \text{ bounded}\}.
\]
For $P,Q \in \Prob(\R)$ then 
\la                                                              \label{Eq:Def_zeta_s}
  \zeta_s(P,Q) &\coloneqq& \sup_{f\in\cF_s^\infty} |Pf -Qf|
\al
is called the {\em Zolotarev distance of order $s$} from $P$ to $Q$, and one further defines a   
\textit{weighted variation distance} as 
\[
  \nu_s(P,Q) &\coloneqq& \int |x|^s\,\dd|P-Q|(x),
\]
which is also called the $s$-th \textit{absolute pseudomoment}~\cite[p.~67]{Zolotarev1986}.
\end{defi}

Let us note that in~\cite[p.~44]{Zolotarev1986} and \cite[p.~100]{Senatov1998}, our $\cF^\infty_s$ is denoted 
by $\cF_s$, and that in these books our $\cF_s$ is implicitly used without any convenient notation.
The latter may have led to some of the clearly existing confusion in the literature.
For example, one finds in several publications,
usually obscured by employing  random variable notation, 
in effect the definition~\eqref{Eq:Def_zeta_s}   with $\cF_s$  in place of $\cF_s^\infty$, 
which makes sense,  and then no difference by the apparantly not completely trivial 
Theorem~\ref{Thm:First_facts_on_zeta_distances}(d) below, iff $P,Q\in\Prob_s(\R)$.
As a recent  example of such an unclear ``definition'' without assuming $P,Q\in\Prob_s(\R)$, we 
can mention \cite[(8), the case of $s=1$, $\mu=\nu$ the standard Cauchy law, 
once $Y=X$ and once $Y=-X$, $f$ the identity]{Neiniger_Sulzbach} 
where, however, the error is immediately admitted.

\begin{Thm}[Well-known facts about $\zeta_s$]                             \label{Thm:First_facts_on_zeta_distances}
Let $s = m+\alpha$ be as in Definition~\ref{Def:zeta-distances}. 

\smallskip{\rm\textbf{(a)}}
For $P,Q\in\Prob(\R)$, the  value of $\zeta_s(P,Q)$  does not change if in the definition of~$\cF_s$
the functions~$f$ are assumed to be $E$-valued rather than $\R$-valued, with $E$ any
Banach space not degenerated to one point.  

\smallskip{\rm\textbf{(b)}} On  $\Prob(\R)$, $\zeta_s$ is an extended metric, that is, a metric except that
it may also assume the value $\infty$.   

\smallskip{\rm\textbf{(c)}} For  $P\in\Prob(\R)$ and $Q \in \Prob_s(\R)$, we have the equivalence chain
\la                                                                            \label{Eq:Equivalences_zeta_s_finite}
  \zeta_s(P,Q) <\infty &\Leftrightarrow& P\in \Prob_s(\R) \text{ and } \mu_j(P)=\mu_j(Q) \text{ for }j\in\{1,\ldots,m\} \\
   &\Leftrightarrow&   P\in \Prob_s(\R) \text{ and } \zeta_s\left(P,Q\right) 
          \le \frac{\Gamma(1+\alpha)}{\Gamma(1+s)}\nu_s(P,Q).   \nonumber 
\al 
Hence, if $c_1,\ldots,c_m\in\R$ are given, then $\zeta_s$ is a metric on the (possibly empty) set 
$\{P\in\Prob_s(\R) :   \mu_j(P)=c_j\text{ for } j \in\{1,\ldots,m\}\}$.
In particular, $\zeta_3$ is a metric on $\widetilde{\cP_3}$.

\smallskip{\rm\textbf{(d)}} Let $P,Q\in\Prob_s(\R)$. Then we may omit the boundedness condition on $f$ 
in the definition~\eqref{Eq:Def_zeta_s}, that is, we have 
\la                              \label{Eq:zeta_s_in_Prob_s}
  \zeta_s(P,Q) &=& \sup_{f\in\cF_s} |Pf -Qf| , 
\al
and we further have  
\la                                                      \label{Eq:Pf-Qf_vs_zeta(P,Q)}
   | Pf -Qf | &\le& \|f^{(m)}\|^{}_{\mathrm{L},\alpha}\ \zeta_s(P,Q) \quad\text{ for } f\in\cC^{m,\alpha}(\R,\R).
\al
\end{Thm} 

References or proofs for Theorem~\ref{Thm:First_facts_on_zeta_distances} are given in section~\ref{Sec:zeta-metrics} on p.\,\pageref{Sec:zeta-metrics}, together with further facts about $\zeta_s$.  With the above preparations, we can state:
\begin{Thm}[essentially Theorem~\ref{Thm:Main} rewritten]               \label{Rem:Main_thm_rewritten}           
Let $n\in\N$ and $P_i, \sigma_i,Q_i,\sigma,\rho_i$ for $i\in\{1,\ldots,n\}$  be as in Theorem~\ref{Thm:Main}. 
Then we have                           
\la           \label{Eq:Main_ineq_rewritten}
   \zeta_3\left( \widetilde{\bigconv_{i=1}^n P_i} \ ,\ \widetilde{\,\bigconv_{i=1}^n Q_i}\,  \right)
  &\le&   \frac{1}{6\sigma^3}\sum\limits_{i=1}^n \sigma_i^3 B(\varrho_i),
\al
with equality whenever each  $P_i$ is  a two-point law and also the centred third moments of the $P_i$ are
all $\ge 0$ or all $\le0$.
\end{Thm}
Indeed, if Theorem~\ref{Thm:Main} is assumed to be true, then applying the  definition of $\zeta_3$ immediately 
yields inequality~\eqref{Eq:Main_ineq_rewritten}, and using also~\eqref{Eq:zeta_s_in_Prob_s} 
from Theorem~\ref{Thm:First_facts_on_zeta_distances}(d) yields the accompanying equality statement.
Conversely, if~\eqref{Eq:Main_ineq_rewritten} is proved, then, using~\eqref{Eq:Pf-Qf_vs_zeta(P,Q)}, we 
get Theorem~\ref{Thm:Main} in the case of $E=\R$ and except for the equality statement.

\begin{Rem}
Under the assumptions of Theorem~\ref{Rem:Main_thm_rewritten}, we have the equivalence 
\la                                                           \label{Eq:Error=0_iff_bound=0}
  \text{L.H.S.\eqref{Eq:Main_ineq_rewritten}} = 0 &\Leftrightarrow& \text{R.H.S.\eqref{Eq:Main_ineq_rewritten}} = 0. 
\al
Here the implication ``$\Leftarrow$'' of course follows trivially  from~\eqref{Eq:Main_ineq_rewritten}. Conversely,
if we have $\text{L.H.S.\eqref{Eq:Main_ineq_rewritten}} = 0$, 
then we get $\widetilde{\bigconv_{i=1}^n P_i} = \widetilde{\,\bigconv_{i=1}^n Q_i}$
and hence,  assuming from now on  without loss of generality the $P_i$ to 
be centred,  and recalling that  $\sigma(P_i)=\sigma_i=\sigma(Q_i)$ for each~$i$, we have
\la                                                \label{bigconv_P_i=bigconv_Q_i} 
  \bigconv_{i=1}^n P_i &=& \bigconv_{i=1}^n Q_i . 
\al   
We now use some well-known elementary facts about  cumulants,
for which we may refer to~\cite{Mattner1999,Hald2000,Mattner2004}.
Cumulants are certain functions $\kappa_\ell:\Prob_\ell(\R)\rightarrow\R$ for $\ell\in\N$,
most importantly $\kappa_1=\mu(\cdot)$, $\kappa_2=\sigma^2(\cdot)$,   
$\kappa_3=\int(x-\mu(P))^3\,\dd P(x)$ for $P\in\Prob_3(\R)$, and
$\kappa_4=\int(x-\mu(P))^4\,\dd P(x) - 3\sigma^4(P)$ for $P\in\Prob_4(\R)$, 
designed to enjoy the additivity
\la                                    \label{Eq:Cumulants_additive}
   \kappa_\ell(P\ast Q) &=&  \kappa_\ell(P) +  \kappa_\ell(Q)\quad\text{ for } \ell\in\N\text{ and }
     P,Q\in\Prob_\ell(\R).
\al
Observing now that, for a centred $P\in\Prob_4(\R)$, we have 
\[
  \kappa_4(P) &=&  \int x^4\dd P(x) -3\sigma^4(P)
 \,\ \ge\,\ \left(\int x^2\dd P(x)\right)^2-  3\sigma^4(P) = -2\sigma^4(P)
\]
with equality throughout iff $P=\frac12\left(\delta_{-\sigma(P)} + \delta_{\sigma(P)}\right)$,
by, say,  Jensen's inequality with the strictly convex square function 
and by centredness of $P$, we get from \eqref{bigconv_P_i=bigconv_Q_i}, 
using~\eqref{Eq:Cumulants_additive} with $\ell=4$ in the first step,
\[
  \sum_{i=1}^n\kappa_4(P_i) &=& \sum_{i=1}^n\kappa_4(Q_i) \,\ =\,\ \sum_{i=1}^n\left(-2\sigma^4(Q_i)\right) 
 \,\ =\,\ \sum_{i=1}^n\left(-2\sigma^4(P_i)\right),
\]
and thus  $P_i=Q_i$ and hence $\varrho(P_i)=1$  
for each $i$, and hence   $\text{R.H.S.\eqref{Eq:Main_ineq_rewritten}} = 0$
due to $B(1)=0$.

Thus the error bound~\eqref{Eq:Main_ineq_rewritten} in Theorem~\ref{Rem:Main_thm_rewritten} enjoys 
the property~\eqref{Eq:Error=0_iff_bound=0} in analogy to classical refinements of the Berry--Esseen bound for normal approximations to convolution products first obtained in the i.i.d.~case, after a preliminary result of Zolotarev~\cite{Zolotarev1965},  by Paulauskas~\cite{Paulauskas_1969}, and then quickly generalized or sharpened in publications up to 1973 by Sazonov~\cite{Sazonov_1972}, Nagaev and Rotar'~\cite{Nagaev_Rotar}, and  Zolotarev~\cite{Zolotarev1973}; reviews by Sazonov~\cite[pp.~9, 68]{Sazonov_1981}, Rotar'~\cite[\S2]{Rotar1982}, Petrov~\cite[pp.~190--191, subsections 5.10.16--5.10.18]{Petrov_1995}, and~ Zolotarev~\cite[section 6.5.1]{Zolotarev1986} point to further relevant works, including several ones by the authors already mentioned here and by Ulyanov, in particular \cite{Ulyanov1976,Ulyanov1978}, 
to which one can add, among others, the papers of Shiganov~ \cite{Shiganov_1987}, Paditz~\cite{Paditz_1988}, and, treating  asymptotic expansions, Yaroslavtseva~\cite{Yaroslavtseva2008}.

In contrast to our bound in~\eqref{Eq:Main_ineq_rewritten}, those refinements have to use some so-called (absolute) pseudo-- or difference--moments instead of ordinary absolute moments of the involved distributions. 
\end{Rem}

\subsection{Normal approximation}              \label{Subsec:NormApprox}
Coming now to the normal approximation results following from Theorem~\ref{Rem:Main_thm_rewritten}, let us first consider 
in Theorem~\ref{Thm:Normal_appr} below the i.i.d.~case. There
\la                                                        \label{Eq:Def_epsilon_n}
 \varepsilon_n \coloneqq \zeta_3 \left(\widetilde{\mathrm{B}_{n,\frac12}},\mathrm{N} \right)\quad\text{ for }  n\in\N
\al
plays the role of a higher order error term, as is made explicit by the following auxiliary result.

\begin{Thm}                                                               \label{Thm:epsilon_n}
For $n\in\N$,  we have, with the first equality to be read from right to left due to the $O(n^{-2})$,
\la                         \label{Eq:Ineq_eps_n}
\qquad  \frac{1}{6\sqrt{2\pi} n} +O \left(\frac1{n^{2}} \right)
  & =& \frac16\left\{\begin{array}{ll}
      \left|\left(2n^{\frac12} +n^{-\frac12}- n^{-\frac32}\right)
        \mathrm{b}^{}_{n,\frac12}(\lfloor\tfrac{n}{2}\rfloor) -\frac{4}{\sqrt{2\pi}}\right|  &\text{ if $n$ is odd},\!  \\
        \left| 2n^{\frac12} \mathrm{b}^{}_{n,\frac12}(\tfrac{n}{2}) -\frac{4}{\sqrt{2\pi}}\right|
  \phantom{\int\limits_0^{\frac11}}   
&\text{ if $n$ is even}\!\!\! 
\end{array}   \right\}       \\ \nonumber
 &=&  \left| \left( \widetilde{\mathrm{B}_{n,\frac12}} -\mathrm{N} \right) \tfrac{|\,\cdot\, |^3}{6}\right|  
  \,\ \le \,\   \varepsilon_n \\   \nonumber
 & < & \frac{1}{3\sqrt{2\pi}n} + \bigg(\frac{4+\zeta(\frac12)}{\sqrt{2\pi}}-1\bigg)\frac1{6n^{3/2}}
\\ \nonumber
&<&\frac{0.1330}{n} +\frac{0.0022}{n^{3/2}} \,\ \le \,\ \frac{0.1352}{n},
\al
where $\zeta(\cdot)$ is the Riemann zeta-function, in particular $\zeta(\frac12)=-1.4603\ldots.$
\end{Thm}

The proof of Theorem~\ref{Thm:epsilon_n} is given in section~\ref{Sec:zeta3-between-N-B} on p.\,\pageref{Page:proof-of-Thm:epsilon_n}.

The above lower bound for $\epsilon_n$ holds even with equality in case of $n=1$, by 
Example~\ref{Example:zeta_3(Q,N)_zeta_4(Q,N)} below, and we conjecture that, in the general case, 
it is at least asymptotically exact.

\begin{Thm}                                              \label{Thm:Normal_appr}
For $P\in\cP_3$ and $n\in\N$, we have
\la                                          \label{Eq:Main_normal}
  \zeta_3\left(\widetilde{P^{\ast n}},\mathrm{N} \right)
  &\le & \frac{B(\rho(P))}{6\,\sqrt{n}} +\varepsilon_n, 
\al
where, on the right, the leading term for $n\rightarrow\infty$ is optimal in the sense of
\la                 \label{Eq:B-E_normal_opt_for_n_large}
 && \frac{B(\rho)}{6} \ =\ \lim_{n\rightarrow\infty} \sqrt{n}\, \zeta_3\left(\widetilde{P_\rho^{\ast n}},\mathrm{N} \right)
  \ =\ \sqrt{k} \left| \widetilde{P_\rho^{\ast k}}f - \mathrm{N}f \right|
  \quad\text{ for }\rho\in\mathopen[1,\infty\mathclose[ \text{ and } k\in\N,
 \al                                         
with $P_\rho\in\cP_3$  being the two-point law defined in~\eqref{Eq:Def_P_rho} and satisfying $\rho(P_\rho)=\rho$, and with $f\in\cF_3$ given by $f(x)=x^3/6$ for $x\in\R$, and the leading term for $\rho\rightarrow1$ is asymptotically exact in the sense of
\la                        \label{Eq::B-E_normal_opt_for_rho_small}     
  \varepsilon_n &=&
  \lim_{P\in\cP_3\colon\rho(P)\rightarrow1} \zeta_3\left(\widetilde{P^{\ast n}},\mathrm{N} \right)
   \quad\text{ for } n\in\N.
\al
\end{Thm}

The proof of Theorem~\ref{Thm:Normal_appr} is given in section~\ref{Sec:zeta3-between-N-B} on p.\,\pageref{Page:proof-of-Thm:Normal_appr}.

\begin{Rem}                                         \label{Remark_1.12}
In view of~\eqref{Eq:Main_normal} and $\epsilon_n=O(n^{-1})$, the first 
equation in~\eqref{Eq:B-E_normal_opt_for_n_large} yields, as an alternative formulation of the 
large $n$ optimality of~\eqref{Eq:Main_normal}:
\la                                                                 \label{Eq:B-E_normal_opt_for_n_large_(alternative)} 
  \frac{B(\rho)}{6} &=& \max_{P\in\cP_3\colon \rho(P)=\rho} \varliminf_{n\rightarrow\infty} 
    \sqrt{n}\, \zeta_3\left(\widetilde{P^{\ast n}},\mathrm{N} \right) \quad\text{ for }\rho\in[1,\infty[,
\al 
with the maximum attained for $P=P_\rho$. We suspect that in~\eqref{Eq:B-E_normal_opt_for_n_large_(alternative)}   
one can replace ``$\varliminf$'' by ``$\lim$'', since if $P\in\cP_3$ is given and if also $f\in\cF_3$ is fixed,
then we have 
\la                                                                        \label{Eq:Edgeworth_for_zeta}
  \lim_{n\rightarrow\infty}\sqrt{n}\left|\widetilde{P^{\ast n}}f - N f\right| &=& |E f| 
\al 
with $E$ denoting here the signed measure on $\R$ with the distribution function
$x\mapsto(1-x^2)\mathrm{e}^{-x^2/2}\mu_3(P)/(6\sqrt{2\pi})$ occurring in the short Edgeworth expansion
for $\widetilde{P^{\ast n}}$, by applying  \cite[Theorem (3.6) in the i.i.d.~case with $k=1$, $s=s_0=3$, 
$p=2$ for $|\alpha|=1$]{GoetzeHipp1978}. However, for an arbitrary $P\in\cP_3$,
we are not aware of a reference conveniently 
yielding the convergence in~\eqref{Eq:Edgeworth_for_zeta} uniformly in $f\in\cF_3$, 
which would then yield the existence of  
$\lim_{n\rightarrow\infty} \sqrt{n}\, \zeta_3\left(\widetilde{P^{\ast n}},\mathrm{N} \right)
= \sup_{f\in \cF_3}| E f|$. In the special case of $P=P_\rho$ this limit exists,
as claimed in \eqref{Eq:B-E_normal_opt_for_n_large},
by the proof of  Theorem~\ref{Thm:Normal_appr}. 
\end{Rem}

\begin{Rem} 
Inequality~\eqref{Eq:Main_normal} often improves Tyurin's 
estimate~\cite[Theorem~4]{Tyurin2009arxiv}, \cite{Tyurin2009DAN}, \cite[Theorem~4]{Tyurin2010TVP} (with~\cite{Tyurin2009arxiv} actually being the 
final one among the three papers)
\la                                                                            \label{Eq:Tyurin_zeta3(noniid)}  
  \zeta_3\bigg(\widetilde{\bigconv_{i=1}^n P_i},\,\mathrm{N} \bigg)
  &\le& \frac1{6\sigma^3}\sum_{i=1}^n\sigma_i^3\varrho_i \quad \text{for } P_1\ldots,P_n\in\cP_3
\al
in the i.i.d.~case, where the latter takes the form
\la \label{Eq:Tyurin_zeta3}  
     \zeta_3(\widetilde{P^{\ast n}},\mathrm{N}) &\le& \frac{\rho(P)}{6\sqrt{n}}\ \quad\text{ for } 
  P\in\cP_3\text{ and } n\in\N
\al
and is optimal in the sense that the constant factor $1/6$ cannot be made less if $\rho(P)$ is allowed to be arbitrarily large. Indeed, in view of $B(\rho)<\rho$ and $\varepsilon_n =O(n^{-1})$, inequality~\eqref{Eq:Main_normal}  improves~\eqref{Eq:Tyurin_zeta3}  for \textit{every} value of $\rho\ge1$ and every sufficiently large $n\in\N$, namely iff
$$
 6\sqrt n \varepsilon_n \,\ <\,\  \rho-B(\rho),
$$
which, by Theorem~\ref{Thm:epsilon_n},  is surely true for
\la                                                           \label{Eq:n>=0.65804/(rho-B(rho))^2}
 n&\ge& \Big(\frac{6\cdot0.1352}{\rho-B(\rho)}\Big)^2 \,\ = \,\  \frac{0.65804\ldots}{(\rho-B(\rho))^2}.
\al
\end{Rem}

Here is a table of the values of $\rho$ and $n$ satisfying condition~\eqref{Eq:n>=0.65804/(rho-B(rho))^2}, where, for convenience, we also provide values of $B(\rho)$ rounded up:
$$
\begin{array}{||c|c|c|c|c|c|c|c|c|c|c|c|c||}
\hline
\rho\le &1.01&1.10&1.18&1.24&1.30&1.52&1.66&1.77&1.94&2.17&2.33& 2.519
\\\hline
B(\rho)\le &0.17&0.53&0.72&0.83&0.94&1.27&1.45&1.59&1.80&2.06&2.24&2.438
\\\hline
n\ge &1&2&3&4&5&10&15&20&30&50&70&100
\\\hline
\end{array}
$$

\begin{Example}
Let $P$ be an exponential distribution. Then $\rho=12e^{-1}-2=2.4145\ldots,$ $B(\rho)=2.3248\ldots,$ and condition~\eqref{Eq:n>=0.65804/(rho-B(rho))^2} holds  for $n\ge82$.

If $P$ is a uniform distribution on an interval, then $\rho=3\sqrt{3}/4=1.2990\ldots,$ $B(\rho)=0.9302\ldots,$ and condition~\eqref{Eq:n>=0.65804/(rho-B(rho))^2} holds  for $n\ge5$.

If $P$ is the Bernoulli distribution with parameter $p\in]0,\frac12]$, then, denoting $q\coloneqq1-p$, we have $\rho(P)=(p^2+q^2)/\sqrt{pq},$ $B(\rho)=(q-p)/\sqrt{pq},$ $\rho-B(\rho)=2p\sqrt{p/q},$ and condition~\eqref{Eq:n>=0.65804/(rho-B(rho))^2} holds for:
\\\mbox{}\qquad\qquad
$n\ge1$ \ if \ $p\ge0.45$,\qquad
$n\ge2$ \ if \ $p\ge0.38$,\qquad
$n\ge3$ \ if \ $p\ge0.34$,\\\mbox{}\qquad\qquad
$n\ge4$ \ if \ $p\ge0.31$,\qquad
$n\ge17$ \ if \ $p\ge0.2$,\qquad
$n\ge149$ \ if \ $p\ge0.1$.\\
In particular, in the symmetric case ($p=1/2$) our bound~\eqref{Eq:Main_normal} is of course
sharper than~\eqref{Eq:Tyurin_zeta3}  for every $n\in\N$.

If $P$ is the Poisson distribution with parameter $\lambda>0$, then:
\\
if $\lambda = 1$ we have $\rho=1.7357\ldots,$ $B(\rho)=1.5448\ldots,$ and~\eqref{Eq:n>=0.65804/(rho-B(rho))^2} holds for  $n\ge19$;
\\
if $\lambda = 2$ we have $\rho=1.6640\ldots,$ $B(\rho)=1.4543\ldots,$ and~\eqref{Eq:n>=0.65804/(rho-B(rho))^2} holds for $n\ge15$;
\\ 
if $\lambda = 4$ we have $\rho=1.6294\ldots,$ $B(\rho)=1.4096\ldots,$ and~\eqref{Eq:n>=0.65804/(rho-B(rho))^2} holds for $n\ge14$;
\\ 
if $\lambda = 8$ we have $\rho=1.6125\ldots,$ $B(\rho)=1.3874\ldots,$ and~\eqref{Eq:n>=0.65804/(rho-B(rho))^2} holds for $n\ge13$.

If $P$ is the geometric distribution with $P_i(\{k\})=p(1-p)^k$ for $k=0,1,2,\ldots,$ then, with $p=0.1,$ we have $\rho=2.4158\ldots,$ $B(\rho)=2.3262\ldots,$ and~\eqref{Eq:n>=0.65804/(rho-B(rho))^2} holds for $n\ge83$.
\end{Example}

Now we present extensions of some of the above results to the non-i.i.d.~case.
\begin{Thm}                                                           \label{Thm:Main(noniid)}
For $P_i,Q_i,\rho_i,\sigma_i,\sigma$ as in Theorem~\ref{Thm:Main} we have
\la                                          \label{Eq:Main_normal(noniid)}
  \zeta_3\bigg( \widetilde{\bigconv_{i=1}^n P_i},\,\mathrm{N} \bigg)
  &\le & \frac{1}{6\sigma^3}\sum\limits_{i=1}^n \sigma_i^3B(\varrho_i)+
\zeta_3\bigg( \widetilde{\bigconv_{i=1}^nQ_i},\,\mathrm{N} \bigg).
\al
Further, if $\sigma_1\ge\sigma_2\ge\ldots\ge\sigma_n$,  then
\[
\zeta_3\bigg( \widetilde{\bigconv_{i=1}^nQ_i},\,\mathrm{N} \bigg)
 &\le& \frac16\bigg(2\sqrt{\frac2\pi}-1\bigg)\frac{\sigma_1^3}{\sigma^3}+ \frac{1}{6\sqrt{2\pi}} \sum_{k=1}^{n-1}\frac{\sigma_{k+1}^3\min\{1,\sqrt n
\sigma_{k+1}/\sigma\}}{\sigma^3\sqrt{k}} \\
& \le & 0.0993\cdot\frac{\sigma_1^3}{\sigma^3}+ 0.0665 \sum_{k=1}^{n-1}\frac{\sigma_{k+1}^3}{\sigma^3\sqrt{k}}.
\]
\end{Thm}

The proof of Theorem~\ref{Thm:Main(noniid)} is given in section~\ref{Sec:zeta3-between-N-B} on p.\,\pageref{Sec:zeta3-between-N-B}.

\begin{Rem}
Inequality~\eqref{Eq:Main_normal(noniid)} improves Tyurin's already optimal bound~\eqref{Eq:Tyurin_zeta3(noniid)}  iff
\[
\zeta_3\bigg( \widetilde{\bigconv_{i=1}^n Q_i},\,\mathrm{N} \bigg)  
 &<& \frac1{6\sigma^3}\sum_{i=1}^n\sigma_i^3\left(\varrho_i -B(\rho_i)\right).
\]
\end{Rem}

Thus, as already indicated at the end of subsection~1.1, Theorems~\ref{Thm:Main} and~\ref{Thm:Main(noniid)} can be regarded as extensions of the results previously obtained in~\cite{Shevtsova2012DAN1,Shevtsova2012DAN2}, \cite[Corollary~4.7 on p.\,284, Theorem~4.13 on p.\,298, Theorem~4.14 on p.\,300, Corollary~4.17 on p.\,302]{Shevtsova2012Debrecen}, \cite[Theorems~2.3, 2.4]{Shevtsova2012TVP} for the uniform metric to $\zeta_3$-metric, so that the inequalities~\eqref{Eq:Main_normal} and~\eqref{Eq:Main_normal(noniid)} can be called \textit{estimates with an asymptotically optimal structure}.

\section{Auxiliary analytic results}

\subsection{Two-point Hermite interpolation, and approximation in $\cF_s$}
The purpose of Lemma~\ref{Lem:Two-point_Hermite_interpolation_polynomials} is to prepare 
through its parts~(c) and~(d) for a proof of Lemma~\ref{Lem:cF_s_properties}, which in turn is used 
in  section~\ref{Sec:zeta-metrics} below  in  our proof of Theorem~\ref{Thm:First_facts_on_zeta_distances}.

\begin{Lem}[On two-point Hermite interpolation polynomials]    \label{Lem:Two-point_Hermite_interpolation_polynomials}
Let $m_0,m_1\in\N_0$, $V_i \coloneqq\R^{\{0,\ldots,m_i\}}$ for $i\in \{0,1\}$, and $V\coloneqq V_0\times V_1$.
For distinct $x_0,x_1\in\R$ and for $y=(y_0,y_1) = \big((y_{0,j})_{j=0}^{m^{}_0},(y_{1,j})_{j=0}^{m^{}_1}\big)\in V$, 
let $p=p^{}_{x_0,x_1,y} =  p^{}_{x_0,x_1,y_0,y_1}$ denote the Hermite interpolation polynomial defined by being a polynomial 
of degree at most  $m_0+m_1+1$ and satisfying  the condition
\la
 p^{(j)}(x_i) &=& y_{i,j} \quad\text{ for }  i\in \{0,1\}\text{ and }j\in\{0,\ldots,m_i\}.
\al

\smallskip{\rm\textbf{(a) Linearity.}} Given distinct $x_0,x_1\in\R$, the map $V \ni y \mapsto p^{}_{x_0,x_1,y}$
is linear with respect to the obvious vector space structures; in particular we have 
$p^{}_{x_0,x_1,y_0,y_1} = p^{}_{x_0,x_1,y_0,0} + p^{}_{x_0,x_1,0,y_1}  = p^{}_{x_0,x_1,y_0,0} + p^{}_{x_1,x_0,y_1,0}$ 
for $y_0\in V_0$ and $y_1 \in V_{1}$.

\smallskip{\rm\textbf{(b) Change of variables.}} For $y \in V$ and distinct  $x_0,x_1\in\R$, we have 
\[ 
 p^{}_{x_0,x_1,y}(x) &=& p^{}_{0,1,z}\big(\frac{x-x_0}{x_1-x_0}  \big) \quad\text{ for }x\in\R 
\]  
with $z\in V$ defined by $z_{i,j}\coloneqq (x_1-x_0)^jy_{i,j}$ for $i\in \{0,1\}$  and $j\in\{0,\ldots,m_i\}$.

\smallskip{\rm\textbf{(c) Positivity.}} Let $-\infty < x_0<x_1<\infty$ and let $(y_0,y_1)\in V$ satisfy 
\la
  y_{0,j} \ \ge \ 0 \text{ for }j\in\{0,\ldots,m_0\},&& (-1)^jy_{1,j} \ \ge \ 0 \text{ for }j\in\{0,\ldots,m_1\}.
\al
Then either $p>0$ on $]x_0,x_1[$, or $y_0=0$, $y_1=0$, $p=0$.

\smallskip{\rm\textbf{(d) Bounds.}} Let $\|\cdot\|$ be a norm on the vector space $V$. 
Then there exists a constant $c = c^{}_{\|\cdot\|} \in \mathopen]0,\infty\mathclose[$ such that the following holds: 
If $y\in V$ and  if $-\infty <x_0<x_1<\infty$, then  
\la                                                      \label{Eq:Bounds_for_interpolation}
 \sup_{x\in [x_0,x_1]} \left| p^{(k)}_{x_0,x_1,y }(x)\right| 
  &\le& c\,\|y\|\,\frac{1\vee |x_1-x_0|^{m_0\vee m_1}}{|x_1-x_0|^k} \quad\text{ for }k\in\N_0.
\al
\end{Lem}
\begin{proof} The existence and uniqueness of $p$ are well-known, and easily imply~(a) and~(b). 

(c) By (a) and (b), the latter applied to $p^{}_{x_0,x_1,y_0,0}$ and also to  $p^{}_{x_1,x_0,y_1,0}$,
we may assume that we have $x_0=0$, $x_1=1$, $y_1=0$. Then the case of $y_0=0$ is trivial, 
and so we assume from now on that at least one coordinate of $y_0$ is even strictly positive, 
and we put 
\[
   k &\coloneqq& \max\{ j\in\{0,\ldots,m_0\} : y_{0,j} >0 \}.
\]
We then have 
\la                                                          \label{Eq:p^{(j)}>0}
    p^{(j)}(x) &>& 0 \quad\text{ for $x>0$ sufficiently close to $0$}
\al 
for $j\in\{0,\ldots,k\}$. 

Assume from now on, to get a contradiction, that we do not have $p>0$ on $]0,1[$. Then, by~\eqref{Eq:p^{(j)}>0}
with $j=0$ and by the intermediate value theorem, 
we have $p(\xi)=0$ for some $\xi\in\mathopen]0,1\mathclose[$. Hence, understanding ``$n$ zeros'' 
to mean ``at least $n$ zeros, counting multiplicity'' in this proof,  $p=p^{(0)}$ has $1+ (m_1+1)=m_1+2$
zeros in $]0,1]$, namely one zero at $\xi$ and $m_1+1$ zeros at $1$.

If now $k\ge 1$ and if $j\in\{0,\ldots,k-1\}$ is such that $p^{(j)}$ has $m_1+2$ zeros in $]0,1]$,
then there is an $\eta=\eta_j\in \mathopen]0,1\mathclose]$ with $p^{(j)}(\eta)=0$ and such that 
$p^{(j)}$ has $m_1+2$ zeros in $[\eta,1]$, and 
then~\eqref{Eq:p^{(j)}>0} with $j+1$ in place of $j$ together with $p^{(j)}(0)\ge0$ implies that 
the maximum of $p^{(j)}$ over $[0,\eta]$ is attained at a point in $\mathopen]0,\eta\mathclose[$,
and hence, in addition applying Rolle's theorem on $[\eta,1]$, we conclude that    $p^{(j+1)}$ has 
$1+(m_1+2-1)=m_1+2$ zeros in $]0,1]$.

The preceding two paragraphs yield that $p^{(k)}$ has $m_1+2$ zeros in $]0,1]$, and 
we have $p^{(k+1)}(0)=\ldots=p^{(m_0)}(0)=0$, with the latter condition of course being empty if $k=m_0$.
Hence $p^{(k+1)}$ has $(m_0-k)+(m_1+2-1)= m_0+m_1 +1-k$ zeros in [0,1]
and is of degree at most $ m_0+m_1 +1-(k+1)=  m_0+m_1 -k$, 
so we have  $p^{(k+1)}=0$ and hence $p$ of degree at most $k\le m_0$, 
yielding   $p(\xi) = \sum_{j=0}^{m_0}y_{0,j}\xi^j/j! >0$, a contradiction.

(d) If $k\ge m_0+m_1+2$, then $p^{(k)} =0$, and then~\eqref{Eq:Bounds_for_interpolation} is trivially 
true even with $c=0$; hence we may assume that $k\in\{0,\ldots,m_0+m_1+1\}$ is fixed in this proof.
Using finite-dimensionality of $V$, we may further assume  that $\|\cdot\| =\|\cdot\|_\infty$, that is,
$\|y\|  = \max_{i,j} |y_{i,j}|$ for $y\in V$, see e.g.~ \cite[pp.~192, 175]{Schwartz1991}. 
Given now $y$ and $x_0,x_1$ as in the claim, we apply~(b) with $z$ as defined there to  get 
\[
  \sup_{x\in [x_0,x_1]}\left| p^{(k)}_{x_0,x_1,y}(x)\right| 
    &=& \sup_{x\in [x_0,x_1]}\left|\tfrac{1}{(x_1-x_0)^k} p^{(k)}_{0,1,z}\big(\tfrac{x-x_0}{x_1-x_0}\big)  \right|
  \,\ \le \,\ \tfrac{c}{(x_1-x_0)^k}\|z\|^{}_\infty  \,\ \le \,\ \text{R.H.S.\eqref{Eq:Bounds_for_interpolation}}, 
\]
where $c$ denotes the norm of the linear map  $V \ni z \mapsto p^{(k)}_{0,1,z}|_{[0,1]} \in \cC([0,1],\R)$, 
with respect to the supremum norms on the two vector spaces, and $c<\infty$ by finite-dimensionality of $V$ 
again, see e.g.~ \cite[p.~279]{Schwartz1991}. 
\end{proof}

We recall the definitions of $\cF_s^\infty$ and $\cF_s$ from Definition~\ref{Def:zeta-distances}.

\begin{Lem}[Denseness of $\cF_s^\infty$ in $\cF_s$]            \label{Lem:cF_s_properties}
Let $s\in\mathopen]0,\infty\mathclose[$ and $f\in \cF_s$. 
Then there exist a sequence $(f_n)$ in $\cF^\infty_s$ and constants $a,b\in[0,\infty[$
with $f_n\rightarrow f$ pointwise and $|f_n| \le a + b|\cdot|^s$ for $n\in\N$. 
If $f=c|\cdot|^s$ with $c\ge 0$, then $(f_n)$ can be chosen to satisfy also $f_n\ge 0$
for $n\in\N$. 
\end{Lem}
\begin{proof} Let $m\in \N_0$ and $\alpha\in\mathopen]0,1\mathclose]$ with $s = m +\alpha$.
We will use the notation of Lemma~\ref{Lem:Two-point_Hermite_interpolation_polynomials}
with $m_1\coloneqq m_2\coloneqq m$.

Let $n\in\N$. We define $y\in V=\R^{\{0,\ldots,m\}}\times\R^{\{0,\ldots,m\}}$ by $y_{0,j}\coloneqq \frac{n-1}{n}f^{(j)}(n)$
and $y_{1,j}\coloneqq 0$ for $j\in\{0,\ldots,m\}$, and we then
apply Lemma~\ref{Lem:Two-point_Hermite_interpolation_polynomials}(d)  with $k\coloneqq m+1$, 
$x_0\coloneqq n$, and $x_1\coloneqq b_n$ with $b_n\ge n+1$ chosen so large that we have
$c\|y\|(b_n-n)^{-\alpha}\le \frac1{2n}$ and hence, by~\eqref{Eq:Bounds_for_interpolation},
so that $p_n \coloneqq p_{n,b_n,y}$ satisfies 
\la                                                              \label{Eq:Bound_p^{(m+1)}}
  \big|p^{(m+1)}_n(x)\big| &\le& \tfrac1{2n}(b_n-n)^{\alpha-1} \quad\text{ for } x\in [n,b_n].
\al
We analogously choose $a_n\le -n-1 $ with $|a_n|$ so large that the polynomial $q_n$ of degree at most
$2m+1$ and with $q_n^{(j)}(a_n)=0$ and  $q_n^{(j)}(-n)=\frac{n-1}{n}f^{(j)}(-n)$ for $j\in\{0,\ldots,m\}$ satisfies 
\la                                                                \label{Eq:Bound_q^{(m+1)}}
  \big|q^{(m+1)}_n(x)\big| &\le& \tfrac1{2n}(-n-a_n)^{\alpha-1}\quad\text{ for } x\in [a_n,-n].
\al
We finally put, using the de Finetti notation introduced in subsection~\ref{Subsec:Notation},      
\[ 
 f_n(x) &\coloneqq& (a_n\le x\le -n)q_n(x)+(|x|<n)\tfrac{n-1}{n}f(x) + (n\le x\le b_n)p_n(x)\quad\text{ for }x\in \R.
\]
Then $f_n\in \cC^m(\R,\R)$ and $f_n$ is bounded. Thus to get $f_n \in \cF_s^\infty$, it remains to prove that
\la                                                                \label{Eq:f_n^{(m)}_alpha-contraction}     
  \sup_{u,\,v\in\R,\,u< v} \frac{|f_n^{(m)}(v)-f_n^{(m)}(u)|}{|v-u|^\alpha} &\le & 1.
\al
So let $-\infty <u<v<\infty$, and let us abbreviate $g\coloneqq f_n^{(m)}$. Then $g(u) = g(u\vee a_n)$
and $g(v)= g(v\wedge b_n)$ and hence $|g(v)-g(u)|/|v-u|^\alpha\le |g(v\wedge b_n)-g(u\vee a_n)|/|v\wedge b_n - u\vee a_n|^\alpha$,
and so we may assume $a_n\le u$ and $v\le b_n$. In the case of $a_n \le u\le -n$ and $n\le v\le b_n$, we use
in the second step below~\eqref{Eq:Bound_p^{(m+1)}} and~\eqref{Eq:Bound_q^{(m+1)}}, 
and also~\eqref{Eq:f_n^{(m)}_alpha-contraction} with $f$ in place of $f_n$, to get 
\[
 |g(v)-g(u)| &\le& |g(v)-g(n)|+|g(n)-g(-n)| + |g(-n)-g(u)| \\
  &\le& \tfrac1{2n}(b_n-n)^{\alpha-1}|v-n| + \tfrac{n-1}{n}|n-(-n)|^\alpha + \tfrac1{2n}(-n-a_n)^{\alpha-1}|-n-u| \\
  &\le& \tfrac1{2n}|v-n|^\alpha + \tfrac{n-1}{n}|n-(-n)|^\alpha +\tfrac1{2n}|-n-u|^\alpha \\
  &\le& |v-n|^\alpha \vee |n-(-n)|^\alpha \vee |-n-u|^\alpha \\
  &\le& |v-u|^\alpha.
\]
The remaining cases needed to prove~\eqref{Eq:f_n^{(m)}_alpha-contraction} are similar or simpler. 

Obviously, $f_n\rightarrow f$ pointwise. Further, by Lemma~\ref{Lem:Two-point_Hermite_interpolation_polynomials}(c), 
we have $f_n\ge 0$ in case of $f=c|\cdot|^s$ with $c\ge 0$. 

Let $g\in\cF_s$. If $s\le 1$, then we have $|g(x)-g(0)|\le |x|^s$ and hence $|g|\le a+b|\cdot|^s$ for $a\coloneqq g(0)$ 
and $b\coloneqq 1$. If $s>1$, then  we have for $x\in \R$ the Taylor formula
\la                                    \label{Eq:Taylor-formula}
  g(x) &=& \sum_{j=0}^{m-1}\frac{g^{(j)}(0)}{j!}x^j
   + \int_0^1 \frac{(1-\lambda)^{m-1}}{(m-1)!}g^{(m)}(\lambda x)x^m\dd \lambda
\al
and get $|g(x)|\le \sum_{j=0}^{m-1} c_j|x|^j + \int_0^1 \frac{(1-\lambda)^{m-1}}{(m-1)!} \left(|g^{(m)}(0)| 
+ |x|^\alpha\right)|x|^m\dd \lambda \le a + b|x|^s$ for certain constants $c_j$ and $a,b$ depending only on   
the availability of bounds for the derivatives up to the order $m$ of $g$ at zero. 
Hence, by the construction of the sequence~$(f_n)$,  we have constants $a,b$ with $|f_n|\le a+b|\cdot|^s$ for each $n$. 
\end{proof}

\subsection{On some special osculatory interpolations and a moment inequality}
Here our goal is the elementary Lemma~\ref{Lem:|x-t|^3<=a+bx+cx^2+d|x|^3}, whose trivial consequence 
Lemma~\ref{Lem:E|X-t|^3<=E(a+bX+cX^2+d|X|^3)} is used in the final Step~7 
of the proof of Theorem~\ref{Thm:Main} in section~\ref{Sec:Main_proofs}.
As for the title of the present subsection, recall that a function $f$ is called first order osculatory at a point 
$x_0$ to a function $g$ if we have $f(x_0)=g(x_0)$ and  $f'(x_0)= g'(x_0)$. 

Let $I\subseteq\R$ be a nondegenerate interval and $s\in\N_0$. Then, following here closely~\cite{PinkusWulbert2005},
a function $f\colon I\to\R$ is said to be $s$-convex on $I$ iff for every choice of $s+1$ pairwise distinct points  $x_0,\ldots,x_s\in I$ the $(s+1)$-st divided difference $[x_0,x_1,\ldots,x_s;f]$ is positive (recall that ``positive'' means $\ge0$, see 
subsection~\ref{Subsec:Notation}). This divided difference may be defined as
$$
[x_0,x_1,\ldots,x_s;f]\coloneqq \frac{U(x_0,\ldots,x_s;f)}{V(x_0,\ldots,x_s)},
$$
where 
$$
U(x_0,\ldots,x_s;f) \coloneqq \left|
\begin{array}{cccc}
1&1&\ldots&1\\
x_0&x_1&\ldots&x_s\\
\vdots&\vdots&\ddots&\vdots\\
x_0^{s-1}&x_1^{s-1}&\ldots&x_s^{s-1}\\
f(x_0)&f(x_1)&\ldots &f(x_s)
\end{array}
\right|,
$$
$V(x_0,\ldots,x_s)\coloneqq U(x_0,\ldots,x_s;(\cdot)^s)=\prod_{i<j}(x_j-x_i)$ is the Vandermonde determinant.  Alternatively one can set~\cite[Chapter\,15]{Kuczma2009}
$$
[x;f]=f(x),\quad [x_0,x_1,\ldots,x_k;f]= \frac{[x_1,\ldots,x_k;f]-[x_0,\ldots,x_{k-1};f]}{x_k-x_0}
  \quad\text{for } k \in\{1,\ldots,s\}.
$$
As  $V(x_0,\ldots,x_s)>0$ for $x_0<x_1<\ldots<x_s,$ a function $f$ is $s$-convex on $I$ iff we have
$U(x_0,\ldots,x_s;f)\ge0$ for all $x_0<x_1<\ldots<x_s\in I.$ Thus, from the definition it immediately follows that a function is $0$-convex iff it is nonnegative, $1$-convex iff it is nondecreasing, 
and $2$-convex iff it is convex in the usual sense. 
Higher order convexity was first  considered by Hopf in his dissertation~\cite{Hopf1926} and was further extensively 
developed by Popoviciu in his thesis~\cite{Popoviciu1933}.

If $ P(x_1,\ldots,x_s;f| \cdot )$  is the unique Lagrange polynomial of degree at most $s-1$ that 
interpolates $f$ at the points $x_1<x_2<\ldots<x_s,$ then~\cite{Popoviciu1933}, \cite[Chapter\,15]{Kuczma2009}
\[
 f(x)- P(x_1,\ldots,x_s;f|x) &=& \frac{U(x_1,\ldots,x_s,x;f) }{V(x_1,\ldots,x_s)} \,\ =\,\   
[x_1,\ldots,x_s,x;f]\prod_{i=1}^s(x-x_i),
\] 
and thus $f$ is $s$-convex on $I$ iff for every choice of $-\infty\eqqcolon x_0<x_1<\ldots<x_s<x_{s+1}\coloneqq+\infty$ we have
\[
 (-1)^{i+s}(f(x)-P(x_1,\ldots,x_s;f|x)) &\ge& 0 \quad \text{ for }  i\in\{0,\ldots,s\}, \
  x\in\mathopen]x_i,x_{i+1}\mathclose[\cap I. 
\]

If $s\ge 2$, then a continuous function $f$ is $s$-convex on $I$ iff on the interior of $I$ the derivative $f^{(s-2)}$ exists and is convex~\cite{Hopf1926,Popoviciu1933,Kuczma2009}. If $f$ is $s$ times differentiable on $I$, then $f$ is $s$-convex iff $f^{(s)}\ge0$ on $I$~\cite{Popoviciu1933,Kuczma2009}.

\begin{Lem}                                        \label{Lem:Osc_with_Laplace convex}
Let $I\subseteq \R$ be an interval, $s,t\in I$ with $s\neq t$, and $f:I \rightarrow \R$ 
twice differentiable with 
\la                              \label{Eq:f_osculates_zero_twice}
  f(s)\,=\,f'(s)\,=\,f(t)\,=\,f'(t)\,=\,0
\al
and $f''$ convex on $I$. Then we have $f\ge 0$ on $I$. If further $u\in I\setminus\{s,t\}$ 
satisfies $f(u)=0$, then we have $f=0$ on the convex hull of $\{s,t,u\}$. 
\end{Lem}

\begin{proof}
The existence of $f''$ and its convexity yield the $4$-convexity of $f$; hence for every choice of 
$t_1,t_2,t_3,t_4\in I$ with $t_1<t_2<t_3<t_4$, the Lagrange interpolation polynomial $p$ of degree $\le3$
with $p(t_j)=f(t_j)$ for each $j$ satisfies for $x\in I$ respectively $f(x)\ge p(x)$ if 
$ t_4 \le x $ or $t_2\le x\le t_3$ or $x\le t_1$, and  $f(x)\le p(x)$ if $t_3 \le x \le t_4$  or $t_1\le x\le t_2$.
This continues to hold if some, but not all, of the $t_j$ coincide and $p$ is accordingly the corresponding Hermite interpolation 
polynomial, in view of the continuous dependence of the latter on $(t_1,t_2,t_3,t_4)$ due to the continuity of $f''$,
compare \cite[p.~119, Theorem~6.3]{DeVore_Lorentz1993}.

To prove now the lemma, we may assume  $s<t$. Assumption~\eqref{Eq:f_osculates_zero_twice} says that 
$p\coloneqq0$ is the Hermite interpolation polynomial of degree $\le 3$ for $f$ and the nodes  
$t_1\coloneqq t_2\coloneqq s$ and $t_3\coloneqq t_4 \coloneqq t$, and hence we get $f\ge 0$ on $I$. 
If further $u$ is as stated, then we prove also $f\leq0$ on the convex hull of $\{s,t,u\},$ by applying the previous paragraph to $p\coloneqq0$, but now with
$(t_1,t_2,t_3,t_4)\coloneqq(u,s,s,t)$ if $u<s$, $\coloneqq$~$(s,u,u,t)$ if $s<u<t$, 
using that then also $f'(u)=0$ due to $f\geq0$ and $f(u)=0$, and finally $\coloneqq$~$(s,t,t,u)$ if $t<u$. 
\end{proof}

\begin{Lem}                                            \label{Lem:|x-t|^3<=a+bx+cx^2+d|x|^3}
Let $f:\R\rightarrow \R$ be differentiable and let $s,t\in\R$ with $|s|\neq |t|$.

{\rm\textbf{(a)}}
There are unique $a,b,c,d\in\R$ such that 
\[
  g(x) &\coloneqq& a+bx+cx^2+d|x|^3 \quad\text{ for }x\in\R
\]
satisfies
\la                                        \label{Eq:Cond_osc_interpolation}
  g(s)=f(s), \ g'(s)=f'(s), \ g(t)=f(t), \ g'(t)=f'(t).
\al

{\rm\textbf{(b)}} If $f$ is a polynomial of degree at most $3$, then $g$ 
is a global upper or lower bound for~$f$. More precisely, 
if $f(x)=A+Bx+Cx^2+Dx^3$ for $x\in\R$, then we have the equivalence chains
\la
 f \le g \text{ on }\R &\Leftrightarrow& d\ge 0
   \,\ \Leftrightarrow\,\ D\cdot(s + t)  \ge 0  ,   \label{Eq:f_le_g}    \\
 f \ge g \text{ on }\R &\Leftrightarrow&  d\le 0
   \,\ \Leftrightarrow\,\ D\cdot(s +  t) \le 0,      \label{Eq:f_ge_g}    
\al
and the inequality between $f$ and $g$ in \eqref{Eq:f_le_g} or~\eqref{Eq:f_ge_g} is strict on 
all of $\R\setminus\{s,t\}$ iff $D\neq 0$ and $st<0$. 
In any case, we  have $ a = A+Da_0,$ $b= B+Db_0,$ $c = C+Dc_0,$ $d = Dd_0,$  where 
\[
  a_0 = \frac{4|st|^3}{(s+t)(s^2+ 4|st| + t^2)},   &\qquad& b_0 = \frac{6s^2t^2}{s^2+ 4|st| + t^2},   \\
  c_0 =-\frac{12s^2t^2}{(s+t)(s^2+ 4|st| + t^2)}, &\qquad&  d_0 = \frac{\left(|s|+|t|\right)^3}{(s+t)(s^2+ 4|st| + t^2)}
\]
in case of $st\le 0$, and  $a_0=b_0=c_0=0$ and $d_0 = \sign(s)=\sign(t) $ in case of $st>0$. 

{\rm\textbf{(c)}} 
If $f(x)=|x-r|^3$  for $x\in\R$, with some $r\in \R\setminus\{0\}$, and if $s =v\cdot\sign(r)$ and $t=-u\cdot\sign(r)$ for some $u,v$ with $u>v\ge0,$ then we have $f\le g$ on $\R$,  and this inequality is strict on $\R\setminus\{s,t\}$ unless $v=0$. 
 More explicitly, 
\begin{equation}                       \label{|x-t|^3<=a+bx+cx^2+d|x|^3}
|x-r|^3 \,\ \le\,\ a + bx + cx^2 + d|x|^3,
\end{equation}
where $a=a_r(u,v),$ $b=b_r(u,v),$ $c=c_r(u,v),$ $d=d_r(u,v),$ with
\la
 a_r(u,v) &=& |r|^3 + \frac{4u^3v^3}{(u-v)(u^2+4uv+v^2)},              \label{Eq:a_r(u,v)_if_v_le_1}\\
 b_r(u,v) &=& -\sign(r)\left(3r^2 + \frac{6u^2v^2}{u^2+4uv+v^2}\right), \label{Eq:b_r(u,v)_if_v_le_1}\\
 c_r(u,v) &=& 3|r| - \frac{12u^2v^2}{(u-v)(u^2+4uv+v^2)},              \label{Eq:c_r(u,v)_if_v_le_1} \\
 d_r(u,v) &=& \frac{(u+v)^3}{(u-v)(u^2+4uv+v^2)}                       \label{Eq:d_r(u,v)_if_v_le_1} 
\al
for $v\le |r|$ and 
$$
a_r(u,v)= |r|\frac{ 6u^{4} v^{2} \!+\!6u^{2} v^{4} \!+\!12 u^{3} v^{2}|r| \! -\!12u^{2} v^{3}|r|\! -\!4u^{3} vr^2 \!-\!4 u v^{3}r^2\! -\!u^{4}r^2 \!-\!v^{4}r^2 \!+\! 6u^{2} v^{2}r^2 }
{\left(u - v\right) \left(u + v\right) \left(u^{2} + 4 u v + v^{2}\right)},
$$
\[
b_r(u, v) &=& 3r\frac{-4u^{2} v^{2}  -4u^{3} v -4u v^{3} -3u^{2} v|r|  +3u v^{2}|r| +u^{3}|r| -v^{3}|r| -4uvr^2}{\left(u + v\right) \left(u^{2} + 4 u v + v^{2}\right)}, \\
c_r(u, v) &=& 3|r|\frac{ u^{4}  +v^{4} -6u^{2} v^{2} -4u^{3} v -4u v^{3} +4u^{3}|r| -4v^{3}|r|  +2u^{2}r^2  +2v^{2}r^2}{\left(u - v\right) \left(u + v\right) \left(u^{2} + 4 u v + v^{2}\right)}, \\
d_r(u, v) &=& \frac{(u-v+2|r|) \left(u^{2} + v^{2} + 4 u v -2 u|r| +2 v|r| - 2r^2\right)}{\left(u - v\right) \left(u^{2} + 4 u v + v^{2}\right)}
\]
for $v>|r|$. Equality in~\eqref{|x-t|^3<=a+bx+cx^2+d|x|^3} is attained at least {\rm(}and at most as well if $v>0${\rm)} at the two  points ${x=-u\cdot\sign(r)}$ and ${x=v\cdot\sign(r)}.$ 
\end{Lem}

\begin{figure}
\includegraphics[width=0.49\textwidth]{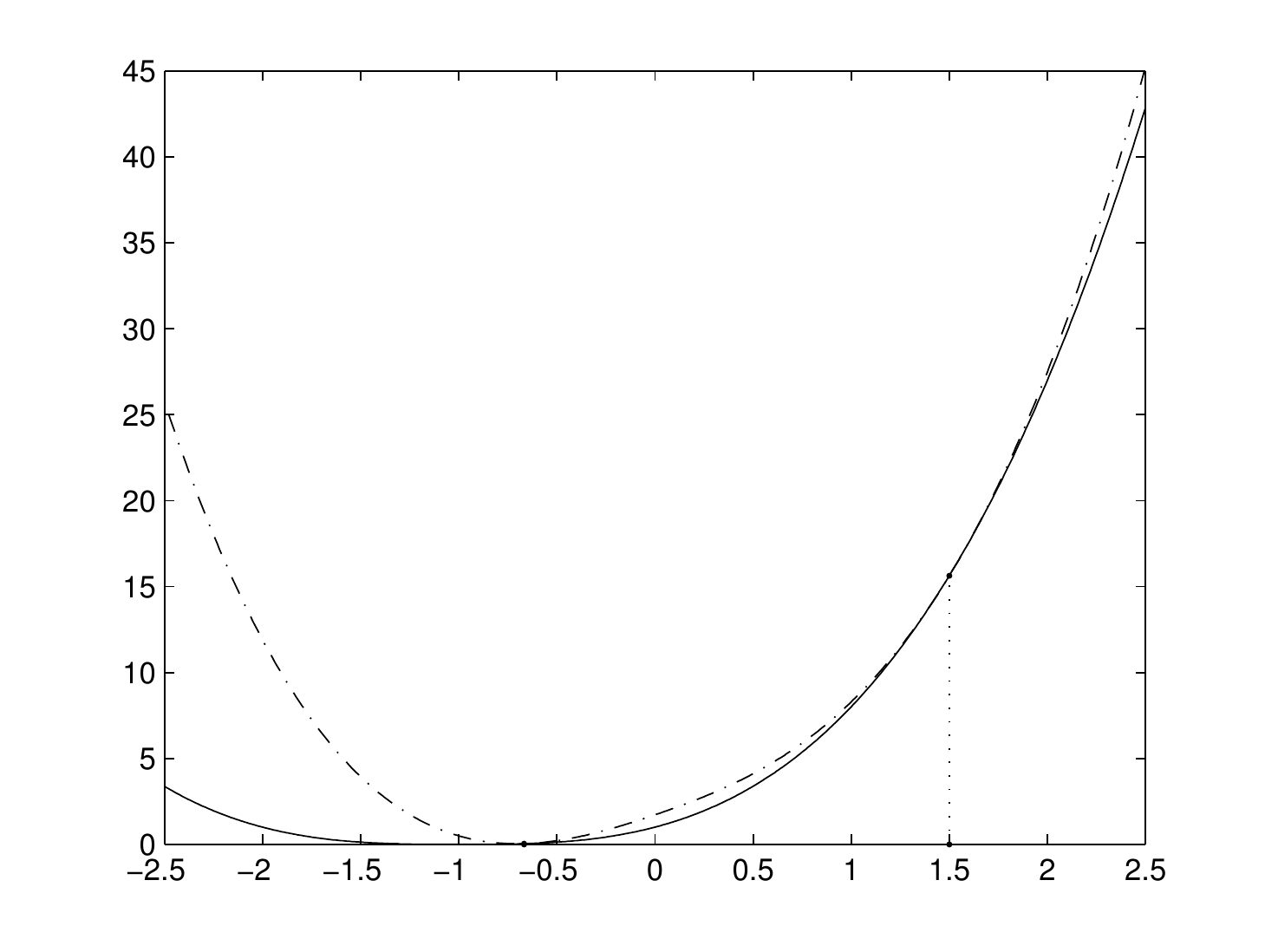}
\includegraphics[width=0.49\textwidth]{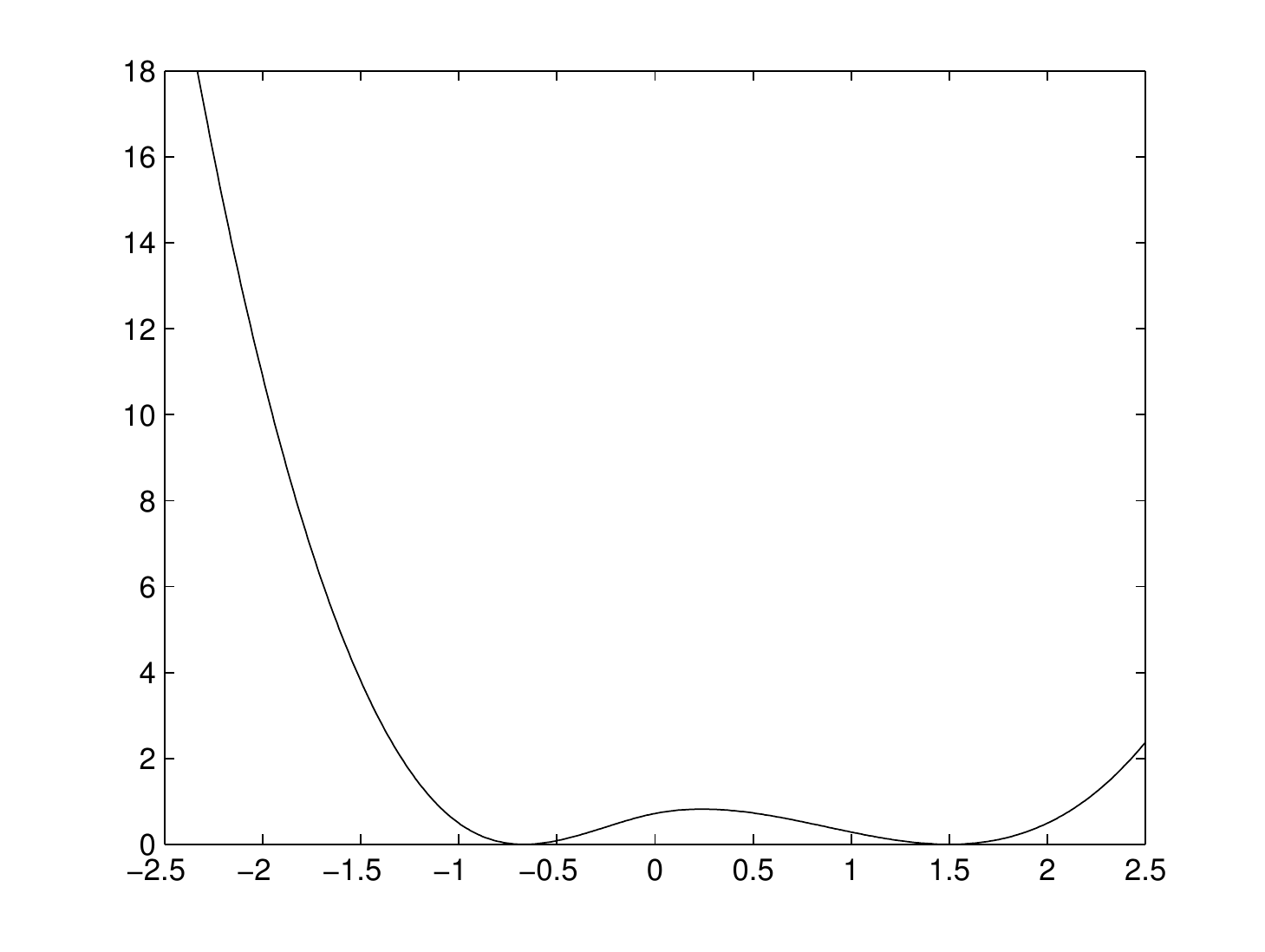}
\caption{Left: plots of the functions $f(x)=|x+1|^3$ (solid line)
and $g(x)=a+bx+cx^2+d|x|^3$ (dashdot line) from
Lemma~\ref{Lem:|x-t|^3<=a+bx+cx^2+d|x|^3}(c) with $u=3/2$, $v=2/3$.
Right: plot of the difference $g(x)-f(x)$.}
\label{Fig:|x-t|^3<=a+bx+cx^2+d|x|^3}
\end{figure}

Using monotonicity of the expectation, Lemma~\ref{Lem:|x-t|^3<=a+bx+cx^2+d|x|^3}(c) trivially yields the following

\begin{lem}                                          \label{Lem:E|X-t|^3<=E(a+bX+cX^2+d|X|^3)}
For every $r\in\R\setminus\{0\},$ $u>v\ge0$ and every $P\in\Prob_3(\R),$ we have
$$
\int|x-r|^3\dd P(x)\,\ \le\,\  a + b \int x\dd P(x) + c \int x^2\dd P(x) + d \int |x|^3\dd P(x),
$$
where the coefficients $a=a_r(u,v),b=b_r(u,v),c=c_r(u,v),$ and $d=d_r(u,v)$ 
are defined in Lemma~\ref{Lem:|x-t|^3<=a+bx+cx^2+d|x|^3}(c), with equality iff the distribution $P$ is concentrated in the two points $v\cdot\sign(r),$ $-u\cdot\sign(r)$.
\end{lem}

\begin{Rem}
Lemma~\ref{Lem:E|X-t|^3<=E(a+bX+cX^2+d|X|^3)} generalizes~\cite[Lemma~2]{Shevtsova2017}, where the stated inequality was proved only in the case of $v>|r|$.
\end{Rem}

\begin{proof}[Proof of Lemma~\ref{Lem:|x-t|^3<=a+bx+cx^2+d|x|^3}] (a) Condition~\eqref{Eq:Cond_osc_interpolation} is a system 
of linear equations for $a,b,c,d$  with the determinant
\[
 \begin{vmatrix}
  1 & s & s^2 & |s|^3 \\
  0 & 1 & 2s  & 3s|s|\\
  1 & t & t^2 & |t|^3 \\
  0 & 1 & 2t  & 3t|t| 
 \end{vmatrix}
 &=&   
 \begin{vmatrix}
  1 & 2s  & 3s|s|\\
  t-s & t^2-s^2 & |t|^3-|s|^3 \\
  1 & 2t  & 3t|t| 
 \end{vmatrix}      \\
 &=&
 \begin{vmatrix}
  t^2-s^2 - 2s(t-s) & |t|^3-|s|^3 -3(t-s)s|s|  \\
  2t-2s            & 3t|t|-3s|s|   
 \end{vmatrix}    \\
  &=& \left(t-s\right)  
\begin{vmatrix}
  t-s  & |t|^3 +2|s|^3 -3ts|s|  \\
  2    & 3t|t|-3s|s|   
 \end{vmatrix}    \\
  &=& (t-s)\left( (t-s)(3t|t|-3s|s|)  -2|t|^3-4|s|^3+6ts|s|    \right)  \\
  &=& (t-s)\left( |t|^3-|s|^3 +3ts|s| -3ts|t|    \right)  \\
  & =&
 \left(t-s\right)\left(|t|-|s|\right)\left(t^2+s^2+|ts|-3ts\right)
  \,\ \neq \,\ 0.
\]

(b) 
Lemma~\ref{Lem:Osc_with_Laplace convex} applied to $g-f$ or to $f-g$ yields the first equivalences in~\eqref{Eq:f_le_g} 
and~\eqref{Eq:f_ge_g}, even without knowing $d$ explicitly. One next easily checks in case of $A=B=C=0$ and $D=1$ 
that the stated formulae for $a,b,c,d$  solve the interpolation problem~\eqref{Eq:Cond_osc_interpolation}.
The case of arbitray $A,B,C,D$ then follows by the linearity of the interpolation operator mapping~$f$
to~$g$ according to part~(a). Using now the explicit formula for $d=Dd_0$, 
one obviously gets  the second~equivalences in \eqref{Eq:f_le_g} and~\eqref{Eq:f_ge_g}.
 
In case of $D=0$ or $st\ge0$, we have $g$ identical to $f$ at least on a half-line.
In case of $D\neq 0$ and $st<0$, the existence of any $u\in \R\setminus\{s,t\}$ with 
$f(u)=g(u)$ would imply by Lemma~\ref{Lem:Osc_with_Laplace convex} that $f=g$ holds 
in some neighbourhood of zero, which implies $D=d=0$, a contradiction to $D\neq 0$.

(c) The case $v>|r|$ is proved in~\cite[Lemma~1]{Shevtsova2017}. Let now $v\le|r|.$ By writing $f(x)=  |r|^3 \left|\frac{x}{-r} + 1\right|^3$ and considering $\frac{x}{-r}$ as the new variable,
we may assume that $r=-1$, that is, $f(x)=|x+1|^3$ for $x\in\R$, and
\la                                                    \label{Eq:v<u_etc}
 - 1 &\le& s \,\ =\,\  -v\,\ \le\,\  0\,\ \le\,\ v \,\ <\,\  t \,\ =\,\  u, \quad v\,\ \le\,\ 1.
\al
Let $\widetilde{f}(x)=(x+1)^3$ for $x\in\R$. 
Since  $ s,t \in [-1,\infty[$ and $f= \widetilde{f}$ on  $[-1,\infty[$,
our present $g$ is also the osculatory interpolation to the polynomial  $\widetilde{f}$.
Hence the present formulae for the coefficients of $g$ follow from part~(b) with $A=D=1$, $B=C=3$, 
and in view of $s+t=u-v>0$ we get from~\eqref{Eq:f_le_g} that 
$f(x) = \widetilde{f}(x)  \le g(x)$ holds for $x\in[-1,\infty[$, 
and in view of  $st=-uv\leq0$ we have 
either
equality iff $x\in\{s,t\}$, or $s=v=0$. 
So, setting 
\[
&& h(x)\, \coloneqq\,  g(x)-f(x) \, =\, 
    a+1+(b+3)x+(c+3)x^2+(1-d)x^3 \quad\text{ for }x\in\mathopen]-\infty,-1\mathclose],
\]
it is enough to prove now $h<0$ on $\mathopen]-\infty,-1\mathclose[$. 

We have
\[ 
  h'(x)\ =\ b+3+2(c+3)x+3(1-d)x^2, &&
  h''(x) \ =\  2(c+3) +6(1-d)x,
\]
and, using $u>v\ge0$ from~\eqref{Eq:v<u_etc} and also~\eqref{Eq:d_r(u,v)_if_v_le_1}, we get 
\[
  W \,\ \coloneqq\,\  (u-v)(u^2+4uv+v^2) \,\ >\,\ 0, && \quad d-1 \,\ = \,\  \frac{2v^2(3u+v)}{W} \,\ \ge  \,\ 0 
\]
and hence, for $x\in\mathopen]-\infty,-1\mathclose[$, using~\eqref{Eq:c_r(u,v)_if_v_le_1} with $r=-1$ in the central step,
and $v\in [0,1]$ from~\eqref{Eq:v<u_etc} in the last,
\[
  h''(x) &\ge &  2(c+3) + 6(d-1) 
   \,\ = \,\ \frac{12}{W}u^2\left( u+3v-2v^2  \right) \,\ >\,\ 0
\]
Thus $h'$ is strictly increasing, and hence we get, for $x\in\mathopen]-\infty,-1\mathclose[$,
$$
h'(x)\,\ < \,\  h'(-1) \,\ =\,\ b-2c-3d \,\ = \,\ 
-\frac{6u^2(1-v)(u+3v+v(u-v))}{W}\,\ \leq\,\ 0,
$$
so that $h$ is strictly decreasing, and we get, again for $x\in\mathopen]-\infty,-1\mathclose[$,
$$
h(x)\,\ >\,\ h(-1) \,\ = \,\ a-b+c+d\,\ =\,\ \frac{2u^2(v-1)^2(2uv+u+3v)}{W}\,\ \ge\,\ 0
$$
as desired.
\end{proof}

\subsection{Sign change counting}
The notation and facts of this subsection are used in the formulation and the proof of 
Theorem~\ref{Thm:CrossingPoints}, which in turn is used in Steps~6 and~7 of the proof of Theorem~\ref{Thm:Main}
in section~\ref{Sec:Main_proofs}. Lemma~\ref{Lem:Ineq_S(f)_S(f')} 
refines~\cite[Lemma~4.2]{DenuitLefevreShaked1998}.
 
For sets $A,B\subseteq \R$ and $n\in \N_0$, we put $A^n_{<}\coloneqq \{x\in A^n: x_1<x_2<\ldots<x_n\}$,
$A^n_{\le}\coloneqq \{x\in A^n: x_1\le x_2\le \ldots\le x_n\}$, and $A\le  B$ $:\Leftrightarrow$
$x\le y$ for every choice of $x\in A$ and $y\in B$, and we define $A<B$ similarly. 

Let now $D\subseteq \R$ and let $f:D\rightarrow \R$ be a function.
Then, with a notation as in~\cite[p.~20]{Karlin1968}, one calls
\[
  S^-(f) &\coloneqq& \sup\{n\in\N_0 : \exists\,x\in D_<^{n+1} \text{ with }f(x_i)f(x_{i+1})<0 
   \text{ for }i\in\{1,\ldots,n\}  \} \\  & \in & \N_0\cup\{\infty\} 
\]
the {\em (possibly infinite)  number of (inequivalent) sign changes} of $f$, and 
the restrictions of $f$ obey the rule
\la                                                                     \label{Eq:Sign_changes_and_restrictions}
  S^-(f|_{A\cup B}) &\le& S^-(f|_{A})+ S^-(f|_{B}) + 1 \quad\text{ for $A,B\subseteq D$ with $A\le B$}.
\al

Let us from now on assume for simplicity that  $D=I$ is an interval. For $n\in\N_0$ then clearly
$S^-(f)=n$ is equivalent to the following condition: There exist a $z =(z_1,\ldots,z_n)\in I^n_\le$ and nonempty
(but possibly one-point) intervals $I_0,\ldots,I_n$ with $\bigcup_{j=0}^n I_j =I$
and such that, for $j\in\{0,\ldots,n\}$, we have $f(x)f(y)\ge 0$ for $x,y\in I_j$, 
but in case of $j\ge 1$ also $\sup I_{j-1} = z_j = \inf I_j$ and $f(x)f(y)<0$ for some
$x\in  I_{j-1}$ and $y\in  I_j$.   If this condition holds, let us call every $z$ as above 
a {\em sign change tuple} of $f$, every entry $z_i$ of such a $z$ a {\em sign change} of $f$, 
and two different sign changes of $f$ {\em inequivalent} if they both occur in one sign change tuple. 
If in addition $f$ is left- or right-continuous, then obviously every such 
$z$ belongs to $I^n_<$ and the corresponding intervals $I_j$ are nondegenerate.  Let us finally  call $f\colon  I\to\R$ {\em lastly positive} if 
we have  $f\ge 0$ on~$I$ or there is an $x_0\in I$ with $f(x_0) >0 $ and $f\ge 0$ on $\mathopen]x_0,\infty\mathclose[\cap I$, 
and {\em essentially lastly positive} if 
we have  $f\ge 0$ Lebesgue-a.e.~on~$I$ or there is an $x_0\in I$ with $f \ge 0 $ Lebesgue-a.e.~on $\mathopen[x_0,\infty\mathclose[\cap I$
and not  $f=0$ Lebesgue-a.e.~on $\mathopen[x_0,\infty\mathclose[\cap I$.

We will need  the following variant of Rolle's theorem.
\begin{Lem}                                                                       \label{Lem:Rolle}
Let $I\subseteq\R$ be an interval and let $f:I\rightarrow [0,\infty[$ be absolutely continuous,
not identically zero, and vanishing in the limit at the boundary points $\inf I$ and $\sup I$.
Then there exist $\xi,\eta\in I$ with $\xi<\eta$ and $f'(\xi) > 0> f'(\eta)$. 
\end{Lem}
\begin{proof} We choose a maximizer $x_0$ for $f$. Then $x_0$ is not a boundary point of $I$,
and we have $\int_{x}^{x^{}_0} f'(t) \dd t =f(x_0)-f(x)>0$ for some $x<x_0$ 
sufficiently close to $\inf I$, and then $f'(\xi) >0$ for some $\xi \in\mathopen]x,x_0\mathclose[$. 
Similarly, $f'(\eta) < 0$ for some $\eta\in    \mathopen]x_0,x\mathclose[$ with some  $x>x_0$ 
close to $\sup I$.  
\end{proof}

\begin{Lem}                                                                \label{Lem:Ineq_S(f)_S(f')}
Let $I$ be a nondegenerate interval, $a=\inf I$, $b=\sup I$, $f:I\rightarrow\R$ be absolutely continuous, and let 
$f':I\rightarrow\R$ be almost everywhere a derivative of $f$. 

\smallskip{\rm\textbf{(a)}} If  $\lim_{t\rightarrow b-}f(t)=0$ and if $f'$ is essentially lastly positive, 
then so is $-f$.

\smallskip{\rm\textbf{(b)}} We have 
\la                                                     \label{Eq:Ineq_S(f)_S(f')}
    S^-(f) &\le &  S^-(f') +1 -\left(\lim_{x\rightarrow a+} f(x)=0 \right)-\left(\lim_{x\rightarrow b-} f(x)=0 \right)
\al 
except when $f=0$ and  $S^-(f')=0$. 

More precisely, if $S^-(f') =n\in\N_0$, then also $m\coloneqq S^-(f)$ is finite, and, if $f$ is not identically zero,
with  $y\in I^m_<$ and $z\in I^n_{\le}$ denoting any sign change tuples of $f$ and $f'$ respectively, 
and with 
\la                        \label{Eq:Def_J_0_..._m}
   J &\coloneqq& \left\{ j\in\{0,\ldots,m\} : 1\le j\le m-1,
        \text{ or }j=0\text{ and }\lim_{x\rightarrow a+} f(x)=0,\right. \\
   & &\phantom{\left\{ j\in\{0,\ldots,m\} :\right.} \left. 
  \text{ or }j=m\text{ and }\lim_{x\rightarrow b-} f(x)=0  \right\}  \nonumber
\al
and $y_0\coloneqq a$ and $y_{m+1}\coloneqq b$, for every $j\in J$, there
is a $k\in\{1,\ldots,n\}$ with $z_k\in\mathopen] y_j , y_{j+1}\mathclose[$.
\end{Lem}
\begin{proof} (a) Obvious from $-f(x)=\lim_{y\rightarrow b-}(f(y)-f(x))=\lim_{y\rightarrow b-}\int_x^yf'(t)\dd t$ for $x\in I$.

(b) It suffices to prove the second claim since, under the stated conditions, 
it yields the existence of an injective function $k(\cdot): J\rightarrow \{1,\ldots,n\}$,
hence $\#J\le n$ and thus~\eqref{Eq:Ineq_S(f)_S(f')}, 
and since the remaining cases of $S^-(f')=\infty$ or $f=0$ are trivial.

So let $S^-(f') =n\in\N_0$, $f$ not identically zero, and $z\in I^n_{\le}$  a sign change tuple of $f'$.
With corresponding intervals $I_0,\ldots,I_n$ as above, we have, for each $j\in\{0,\ldots,n\}$,
either $f'\le 0$ on $I_j$ or $f'\ge 0$ on $I_j$, and hence $S^{-}(f|_{I_j})\le 1$, 
and hence $m\coloneqq S^-(f) \le 2n+1 <\infty$, by 
applying~\eqref{Eq:Sign_changes_and_restrictions} $n$ times.  
So let $y\in I^m_<$ be a sign change tuple of $f$, let $J$ be defined by~\eqref{Eq:Def_J_0_..._m},
$y_0\coloneqq a$, $y_{m+1}\coloneqq b$, and let $j\in J$. 
Applying Lemma~\ref{Lem:Rolle} to $f|_{]y_j,y_{j+1}[}$ or its negative 
yields  a $k$ as claimed.
\end{proof}

\subsection{Partial sums of reciprocals of square roots}
As usual, the symbol $\zeta$ without any subscript denotes the Riemann zeta-function.
In particular, $\zeta(\frac12)$ is a negative number as indicated in~\eqref{Eq:zeta(1/2)} below,
see \texttt{oeis.org/A059750} in \cite{Sloane2010}.

\begin{Lem}   \label{Lem:zeta(1/2)}
For $n\in \N=\{1,2,\ldots\}$, we have
\begin{equation}                            \label{Eq:zeta(1/2)}
 \sum_{k=1}^{n-1}\frac1{\sqrt k} -2 \sqrt{n} \,\ <  \,\ \zeta(\tfrac12)
       \,\ = \,\ -1.46035\ldots,
\end{equation}
with equality in the limit as $n\to\infty.$
\end{Lem}
\begin{proof} Let $a_n$ denote the left hand side of the inequality in~\eqref{Eq:zeta(1/2)}.
Then $a_n-a_{n+1} = 2\sqrt{n}\left(\sqrt{1+\frac 1n} -\left(1+\frac1{2n} \right)\right)<0$
by the tangent bound at $1$  for the concave function $\sqrt{\cdot}$.  Hence the
sequence $(a_n)_{n\ge 1}$ is strictly  increasing. Since we have
$\lim_{n\rightarrow\infty} a_n = \zeta(1/2)$ by
\cite[p.~333, (13.10.7) with $s=\sigma=\frac12$]{Hardy1949}, or
see~\cite[p.~192, (4.1)]{Wirths2014}   
for a more elementary proof, the inequality in~\eqref{Eq:zeta(1/2)} follows.
\end{proof}

\section{On few-point reduction theorems}               \label{Sec:reduction-thms}
In this section, we recall some reduction theorems partially used below, with apparently  some novelty in 
part (b) of the first one. For Tyurin's Theorem~\ref{Thm:Tyurin_reduction}
we provide a proof perhaps more natural than the original one.

The term ``component'' below is meant in the usual topological sense
of ``maximal connected subset'',  here of a subset $M$ of $\R^k$.

\begin{Thm}[essentially Richter 1957, \cite{Richter}]              \label{Thm:Richter}
Let $P$ be a law on the measurable space $(\mathcal{X},\mathcal{A})$, let $k\in \N$, and let $f_1,\ldots,f_k$
be real-valued and $P$-integrable functions on $\mathcal{X}$.

{\rm\textbf{(a)}}
There exists a law $Q$ on $(\mathcal{X},\mathcal{A})$ concentrated in $k+1$ or fewer points
such that  $Pf_i =Qf_i$ holds for each $i\in\{1,\ldots,k\}$.

{\rm\textbf{(b)}} Assume in addition that $M\coloneqq\{ (f_1(x),\ldots, f_k(x)) : x\in \mathcal{X}\}$
has at most $k$ components. Then conclusion (a) holds with ``$k$ or fewer'' in place of ``$k+1$ or fewer''.
\end{Thm}
\begin{proof} Let $F(x) \coloneqq (f_1(x),\ldots,f_k(x))$ for $x\in\mathcal{X}$,
so that $M$ as defined in \textbf{(b)} above is the image of the function~$F$,
and let~$C$ denote the convex hull of $M$. Then we have $y\coloneqq \int F \dd P \in C$,
by part of the multivariate Jensen inequality as in
\cite[p.~74, Lemma 3]{Ferguson}  or \cite[p.~348, Theorem 10.2.6]{Dudley}, noting that the measurability condition imposed on $C$
in the second reference is not used anywhere in the proof.

(a) By the Carath\'eodory theorem \cite[p.~29, Theorem~1.3.6]{Hiriart-Urruty_Lemarechal},
the point $y$ is a convex combination of
$k+1$ or fewer points in $M$, that is, there exist not necessarily distinct
$x_1,\ldots,x_{k+1}\in\mathcal{X}$ and $p_1,\ldots, p_{k+1}\in[0,1]$ with $\sum_{j=1}^{k+1} p_j=1$
and $y = \sum_{j=1}^{k+1} p_j F(x_j)$,
that is, $P f_i = Q f_i$ holds for $Q = \sum_{j=1}^{k+1} p_j \delta_{x_j}$ and each $i\in\{1,\ldots,k\}$.

(b) Under the additional hypothesis, the Fenchel-Bunt
refinement~\cite[p.~30, Theorem~1.3.7, see also pp.~245--246]{Hiriart-Urruty_Lemarechal}
of the Carath\'eodory theorem yields that  $y$ is a convex combination of $k$ or fewer points in $C$,
and we conclude as before.
\end{proof}

For $\mathcal{X}$ a Borel subset of $\R$, Theorem~\ref{Thm:Richter}(a)
is contained in \cite[p.~153, Satz~4]{Richter}.  For $\mathcal{X}$ an interval in $\R$
and for the special case of continuous $f_i$, in which case $M$ is connected,
Theorem~\ref{Thm:Richter}(b) is~\cite[p.~153, Satz~5]{Richter}, whereas  in our version and say in
case of $k\ge 3$, one of the functions $f_i$ could for example be an indicator of a subinterval of
$\mathcal{X}$, since then, assuming the remaining functions to be continuous,  
 $M$ would have at most three  components.
For a general measurable space $(\mathcal{X},\mathcal{A})$, Theorem~\ref{Thm:Richter}(a) is stated
in~\cite{Kemperman}, where also further references are given.

In the course of the proof of our main result below, Theorem~\ref{Thm:Richter}(a) allows us to
restrict attention to 5-point laws, which are still rather complex objects.
Using instead  Theorem~\ref{Thm:Richter}(b) would permit us
to consider only 4-point laws. However, the following generalization of a
result~\cite[p.~269, Theorem 2.1 with $n=1$]{Hoeffding1955}
of Hoeffding from 1955,
combined with Theorem~\ref{Thm:Richter}(a) and with the concavity
of the function $B$ from~\eqref{Eq:Def_B(rho)_new}, 
allows a reduction to 3-point laws, which turn out to be sufficiently tractable analytically.
Let us remark that using just Hoeffding's result would again only lead to a reduction to
4-point laws.

For the rest of this section all laws considered are finitely supported and are hence
for notational simplicity regarded as defined on the power set of the basic set $\mathcal{X}$.

\begin{Thm}[implicitly Hoeffding 1955, \cite{Hoeffding1955}]              \label{Thm:Extreme_point_representation}
Let $\mathcal{X}$ be a set, let $k\in \N$, and let $f_1,\ldots,f_k$ be real-valued functions on
$\mathcal{X}$. Then every finitely supported law $P$ on $\mathcal{X}$ is a finite convex combination
$\sum_{j=1}^n\lambda_j P_j$ of laws $P_j$ each concentrated on $k+1$ or fewer support points
of $P$ and satisfying $P_jf_i = Pf_i$ for each $i\in\{1,\ldots,k\}$.
\end{Thm}
\begin{proof} Replacing $\mathcal{X}$
by $\{x\in\mathcal{X} : P(\{x\}) >0\}$, we may assume that $\mathcal{X}$
is finite and is the set of all support points of $P$. Then
\[
 K &\coloneqq& \{ Q\in \Prob(\mathcal{X}) : Qf_i= Pf_i\text{ for }i\in\{1,\ldots,k\}\}
\]
is a convex and compact subset of the finite-dimensional vector space of all $\R$-valued measures
on $\mathcal{X}$, with $P\in K$. Hence, by Minkowski's theorem
\cite[p.~42, Theorem~2.3.4]{Hiriart-Urruty_Lemarechal}, $P$ is a finite convex combination
$\sum_{j=1}^n\lambda_j P_j$ of extreme points $P_j$ of $K$, and then
each $P_j$ is concentrated in at most $k+1$ points:

Indeed, suppose that $Q =\sum_{x\in\mathcal{X}}q_x \delta_x \in K$ is such that its set of support points
$\mathcal{X}_0:=\{x \in \mathcal{X} : q_x >0 \}$ contains at least $k+2$ elements. Then 
\[
  \left\{ r \in \R^{\mathcal{X}_0} : \sum_{x\in \mathcal{X}_0} r_x = 0,
     \sum_{x\in \mathcal{X}_0} r_x f_i(x) = 0 \text{ for }i\in\{1,\ldots,k\}
  \right\}
\]
is a subspace of dimension at least $1$ of $\R^{\mathcal{X}_0}$, hence contains a nonzero $r$,
so that we have
\[
  Q_{\pm} &\coloneqq& Q \pm \epsilon \sum_{x\in\mathcal{X}_0} r_x\delta_x
  \,\ \in \,\   K\setminus \{Q\}
\]
for some $\epsilon>0$, and $Q = \frac12(Q_++Q_-)$. Thus $Q$ is not an extreme point of $K$.
\end{proof}

\begin{Thm}[Tyurin 2009, \cite{Tyurin2009arxiv,Tyurin2009DAN,Tyurin2010TVP}]            \label{Thm:Tyurin_reduction}
Let $\mathcal{X}$ be a set,  $k\in \N$,   $f_1,\ldots,f_k$ real-valued functions
on $\mathcal{X}$, $c_1,\ldots, c_k \in \R$, and
\[
  \mathcal{P} &\coloneqq& \left\{ P \in \Prob(\mathcal{X}) :  \#\mathrm{supp\,} P <\infty,
       \,   Pf_i=c_i \text{ for }i\in\{1,\ldots,k\} \right\}.
\]
Let $F: \cP \rightarrow \overline{\R}$ be quasi-convex, that is, satisfying 
$F(\lambda P + (1-\lambda)Q) \le \max\{F(P),F(Q)\}$ for $P,Q \in \cP$ and $\lambda \in[0,1]$.  Then
\[
  \sup\{F(P) :  P\in\cP  \}
    &=&  \sup\{F(P) : P\in\cP, \,     \#\mathrm{supp\,} P \le k+1  \}.
\]
\end{Thm}
\begin{proof} Applying the representation $P = \sum_{j=1}^n\lambda_j P_j$ from Theorem~\ref{Thm:Extreme_point_representation}, 
and the quasi-convexity condition on $F$ extended by induction, immediately yields the claim.
\end{proof}

Let us finally mention \cite{Winkler,Pinelis_2006}  as starting points for some 
more sophisticated results related to this section.

\section{Auxiliary results for Zolotarev's $\zeta$-metrics}             \label{Sec:zeta-metrics}

\begin{proof}[Proof of Theorem~\ref{Thm:First_facts_on_zeta_distances}]
(a) An obvious Hahn-Banach argument, as in Step~2 of the proof of Theorem~\ref{Thm:Main} in section~\ref{Sec:Main_proofs}
below.

\smallskip(b) Definiteness of $\zeta_s$, that is, the implication $\zeta_s(P,Q)=0 \Rightarrow P=Q$,
is of course very well-known, for example as a consequence of the uniqueness theorem for characteristic functions.
The remaining claims are obvious.

\smallskip(d)   Relation~\eqref{Eq:zeta_s_in_Prob_s} follows from Lemma~\ref{Lem:cF_s_properties}
using dominated convergence. Inequality~\eqref{Eq:Pf-Qf_vs_zeta(P,Q)} follows from~\eqref{Eq:zeta_s_in_Prob_s} 
using the linearity of expectations.

\smallskip(c) If $\zeta_s(P,Q)<\infty$, then we apply Lemma~\ref{Lem:cF_s_properties} 
to $f\coloneqq  \left(\prod_{j=0}^{m-1}(s-j)\right)^{-1} |\cdot|^s\in\cF_s$ 
to get $\infty > \zeta_s(P,Q)\ge |Pf_n -Qf_n|\rightarrow |Pf-Qf|$ 
using dominated convergence for $Qf_n$,  dominated convergence for $Pf_n$ in case of $Pf<\infty$,
and Fatou's Lemma for $Pf_n$ in case of $Pf=\infty$,
and we conclude that $Pf<\infty$, that is 
$P\in\Prob_s(\R)$; and for $j\in\{1,\ldots,m\}$ and $n\in\N$ then~\eqref{Eq:zeta_s_in_Prob_s}  
from part (d) applies to the monomial  $n(\cdot)^j\in\cF_s$, 
and letting $n\rightarrow\infty$ yields $\mu_j(P)=\mu_j(Q)$. 
If the second condition in~\eqref{Eq:Equivalences_zeta_s_finite} holds,
then the third follows easily using~\eqref{Eq:Taylor-formula}, compare \cite[pp.~102--103]{Senatov1998}.
Finally, the third condition in~\eqref{Eq:Equivalences_zeta_s_finite}
implies the first, in view of $\nu_s(P,Q)\le \nu_s(P)+\nu_s(Q)$.   
The remaining claims follow obviously.
\end{proof}

Let us  next  recall  two further well-known properties of $\zeta_s$, 
with $s\in \mathopen]0,\infty\mathclose[$ arbitrary, needed below. 
The first is its {\em regularity}
\la                                          \label{Eq:regularity_of_zeta}
   \zeta_s(P\ast R , Q\ast R) &\le & \zeta_s(P,Q)\quad\text{ for }P,Q,R\in\Prob(\R)    
\al
proved e.g.~in~\cite[p.~101]{Senatov1998}, which, given Theorem~\ref{Thm:First_facts_on_zeta_distances}(b), 
is equivalent to its {\em semiadditivity} 
\la                                      \label{Eq:semiadditivity_of_zeta}
   \zeta_s\big(\bigconv_{i=1}^n P_i , \bigconv_{i=1}^n Q_i \big) &\le& \sum_{i=1}^n\zeta_s(P_i,Q_i)
  \quad\text{ for } n\in \N \text{ and } P_i,Q_i \in\Prob(\R), 
\al
compare~\cite[p.~48]{Senatov1998}.  
To formulate the second, we use here, as well as later in some proofs, the obvious random variable notation 
$\zeta_s(X,Y) \coloneqq \zeta_s(P,Q)$ if $X,Y$ are $\R$-valued r.v.'s with $X\sim P$ and $Y\sim Q$.
Then we have the {\em homogeneity}
\la                                                    \label{Eq:homogeneity_of_zeta}
  \zeta_s(aX,aY)  &=&  a^s\zeta_s(X,Y)\quad\text{ for $a\in[0,\infty[$ and $\R$-valued r.v.'s $X$ and $Y$,}
\al 
the obvious proof of which being given in~\cite[p.~102]{Senatov1998}.

The following Lemma, which is presented in \cite[pp.~108-112]{Senatov1998} without explicit constants, 
allows us in the proof of Theorem~\ref{Thm:Main(noniid)},  
in a case where $aX\sim P$ and $aY\sim Q$ with  small~$a$, 
to use the homogeneity~\eqref{Eq:homogeneity_of_zeta} with a better exponent
than possible by just using~\eqref{Eq:regularity_of_zeta}.
We recall that $\mathrm{N}_\sigma$ denotes the centred normal law on $\R$ with variance $\sigma^2$.  
\begin{Lem}\label{LemZeta_s<=Zeta_s+l}
Let $P,Q\in\Prob(\R)$ and  $s,t,\sigma\in\mathopen]0,\infty\mathclose[$.
Then we have 
\la                                                                              \label{Eq:zet_s_zeta_s+t}
 \zeta_s(P\ast\mathrm{N}_\sigma ,Q\ast\mathrm{N}_\sigma) &\le& C_{s,t} \frac{\zeta_{s+t}(P,Q)}{\sigma^t}
\al
with the finite constant $C_{s,t}$ defined as follows: Writing
\[
  s = \ell +\alpha,\quad t=m+\beta\quad \text{ with $\ell,m\in\N_0$ and $\alpha,\beta \in \mathopen]0,1 \mathclose]$}
\] 
and letting $\varphi$ denote the standard normal density, we put
\[
D_k &\coloneqq& \int|\varphi^{(k)}(x)|\dd x, \qquad 
D_{k,\alpha} \,\ \coloneqq \,\  \int |x|^\alpha\, |\varphi^{(k)}(x)|\dd x   
  \qquad\text{ for } k\in\N_0,    \\
C_{s,t} & \coloneqq&
\left\{
  \begin{array}{ll}
D_m^{\frac{1-\alpha-\beta}{1-\alpha}}\cdot D_{m+1,\alpha}^{\frac{\beta}{1-\alpha}} &\text{ if } \alpha+\beta\le1,\\[4mm]
D_{m+1}^{\frac{\alpha+\beta-1}\alpha}\cdot (2 D^{}_{m+1,\alpha})^{\frac{1-\beta}\alpha} &\text{ if } \alpha+\beta>1.
  \end{array}
\right.  
\]
In particular, if $t\in\N$, hence $m=t-1,$ $\beta=1,$ and $\alpha+\beta>1,$ then
$ C_{s,t} = D_{m+1} = D_t = \int|\varphi^{(t)}(x)|\dd x$, and the first few of these constants
can be explicitly computed, for example
\[
 C_{s,1} &=&  \int|\varphi'(x)|\dd x\,\ =\,\ \frac{2}{\sqrt{2\pi}}\, ,
 \qquad C_{s,2}  \,\ = \,\  \int|\varphi''(x)|\dd x\,\ =\,\ \frac{4}{\sqrt{2\pi\mathrm{e}}}\,.
\]
\end{Lem}
\begin{proof} We shall follow the outline of the reasoning employed 
in~\cite[Lemma 2.10.1]{Senatov1998}.             \label{PageRef:Senatov1998Citation}  
Let $\varphi^{}_\sigma(x)\coloneqq \sigma^{-1}\varphi(x/\sigma)$ for $x\in \R$.
Given any $f\in\cF^\infty_s$, and writing 
\la                                                        \label{Eq:Def_g_h_from_f}
 g(x)\,\ \coloneqq \,\ \int f(x+z)\varphi_\sigma(z)\dd z &\text{ and }&
  h(x) \,\ \coloneqq \,\ \frac{\sigma^t g(x)}{C_{s,t}} \quad \text{ for } x\in \R,
\al
it is sufficient to prove that $h\in \cF^\infty_{s+t}$, for then we would get
\[
 |(P\ast\mathrm{N}_\sigma)f-(Q\ast\mathrm{N}_\sigma)f| &=& |Pg-Qg|
  \,\ =\,\ \frac{C_{s,t}}{\sigma^t}|Ph-Qh|  \,\ \le \,\  \text{R.H.S.\eqref{Eq:zet_s_zeta_s+t}}
\]
as desired. So let $f\in\cF^\infty_s$ and let $g$ and $h$ be defined through~\eqref{Eq:Def_g_h_from_f}.
Then $h$ is obviously bounded, and, with 
\[
  n \,\coloneqq \,  \lceil s+t-1\rceil \, =\, \begin{Bmatrix} \ell+m \\ \ell+m+1 \end{Bmatrix}
 \text{ if } \alpha+\beta  \,\begin{Bmatrix} \le \\ > \end{Bmatrix}\, 1 
 &\text{ and }&  \gamma \, \coloneqq \, s+t-n \, \in\, \mathopen]0,1\mathclose], 
\]
it remains to prove that  we have
\la                                                  \label{g_n_incr-bound_to_prove}
  \left|g^{(n)}(x) - g^{(n)}(y) \right| &\le& \frac{C_{s,t}}{\sigma^t} |x-y|^\gamma \quad\text{ for }x,y\in\R.
\al

If $k\in\N_0$ with $k\ge\ell$, then we obtain, for $x,y\in\R$,  
\la                                                                 \label{Eq:g^(ell)_under_integral}
 g^{(\ell)}(x) &=& \int f^{(\ell)}(x+z)\varphi_\sigma(z)\dd z \,\ =\,\ \int f^{(\ell)}(z)\varphi^{}_\sigma(x-z)\dd z, \\
 g^{(k)}(x)   &=&   \int f^{(\ell)}(z)\varphi^{(k-\ell)}_\sigma(x-z)\dd z  
         \,\ =\,\ \int f^{(\ell)}(x-z)\varphi^{(k-\ell)}_\sigma(z)\dd z,                       \label{Eq:g_nth_deriva} \\
 \quad |g^{(k)}(x)\! -\!g^{(k)}(y)| &\le& \int\! \left| f^{(\ell)}(x\!-\!z)-f^{(\ell)}(y\!-\!z)\right| \left|\varphi^{(k-\ell)}_\sigma(z)\right| \dd z 
  \,\ \le \,\  |x\!-\!y|^\alpha \frac{D_{k-\ell}}{\sigma^{k-\ell}}  \label{g_n_first_incr-bound} 
\al
where, to justify differentiation under the integral, we may in~\eqref{Eq:g^(ell)_under_integral} 
apply the dominated convergence theorem successively using polynomial bounds
on the derivatives $f',\ldots,f^{(\ell)}$, compare~\eqref{Eq:Taylor-formula} and the ensuing line,
and we may treat~\eqref{Eq:g_nth_deriva} similarly, or remember it as a well-known special 
case of the differentiability of Laplace transforms, see for example~\cite[Example]{Mattner2001};
at the last step in~\eqref{g_n_first_incr-bound} we used $f\in\cF_s^{}$ and the change of 
variables $z\mapsto \sigma z$.  Specializing~\eqref{Eq:g_nth_deriva} to $k\coloneqq \ell+m+1$ 
and using at the first step below  $\int \varphi^{(m+1)}_\sigma(z)\dd z =0$ yields
\la                                                                         \label{Eq:g_derivative_bound}
 |g^{(\ell+m+1)}(x)| &=& \left| \int \left(f^{(\ell)}(x-z) - f^{(\ell)}(x)\right)\varphi^{(m+1)}_\sigma(z)\dd z \right| \\
  & \le & \int |z|^\alpha \left|\varphi^{(m+1)}_\sigma(z)\right| \dd z  \,\ = \,\ \frac{D_{m+1,\alpha}}{\sigma^{m+1-\alpha}}   \quad\text{  for }x\in\R.
   \nonumber
\al
 
Let us now first assume that we have $\alpha+\beta\le 1$, and hence $n=\ell+m$ and $\gamma=\alpha+\beta$. 
Then, using~\eqref{Eq:g_derivative_bound} at the second step below, we get 
\[
   \text{L.H.S.\eqref{g_n_incr-bound_to_prove}} &\le&  \|g^{(n+1)}\|^{}_\infty \cdot|x-y| 
  \,\ \le \,\ \frac{D_{m+1,\alpha}}{\sigma^{m+1-\alpha}}|x-y| \quad\text{ for }x,y\in\R,
\]
and taking a geometric mean of this bound and the one from~\eqref{g_n_first_incr-bound} with $k\coloneqq n$, 
with the exponents $u\coloneqq \beta/(1-\alpha)\in\mathopen]0,1\mathclose]$ and $1-u$, yields~\eqref{g_n_incr-bound_to_prove} in the present case.

Let us finally assume that we have $\alpha+\beta>1$, and hence $n=\ell+m+1$ and $\gamma=\alpha+\beta-1$. 
Then, applying below~\eqref{Eq:g_derivative_bound} to $x$ and to $y$, we get   
\[
  \text{L.H.S.\eqref{g_n_incr-bound_to_prove}} &\le& \frac{2 D_{m+1,\alpha}}{\sigma^{m+1-\alpha}}
   \quad\text{ for }x,y\in\R,
\]
and taking a geometric mean of this bound and the one from~\eqref{g_n_first_incr-bound} with $k\coloneqq n$,
with the exponents $v\coloneqq 
(1-\beta)/\alpha\in [0,1[$ and $1-v$,  yields~\eqref{g_n_incr-bound_to_prove} again.
\end{proof}

In Steps~6 and~7  of our proof of Theorem~\ref{Thm:Main}, we will use  Theorem~\ref{Thm:CrossingPoints}
stated below, which collects or refines results known from~\cite{Zolotarev1986},
\cite{DenuitLefevreShaked1998}, and~\cite{BoutsikasVaggelatou2002}. 
In particular, Theorem~\ref{Thm:CrossingPoints}(b) contains~\cite[Theorems~3.3 and~4.3]{DenuitLefevreShaked1998} 
and \cite[p.~353, first part of Theorem~2]{BoutsikasVaggelatou2002}, and adds a converse to the latter,
while Theorem~\ref{Thm:CrossingPoints}(c,d) seems to  be new. 

Let us first recall the definition of the $s$-convex order of laws on $\R$ in accordance 
with~\cite[p.~590]{DenuitLefevreShaked1998},\label{PageRef:DenuitLefevreShaked1998citation} \cite[p.\,351]{BoutsikasVaggelatou2002}, 
\cite[p.~39, Definition~1.6.2~a)]{MuellerStoyan2002}, and~\cite[p.~139]{ShakedShanthikumar2007}, 
but being here somewhat more explicit with respect to the appropriate integrability assumptions:
If  $s\in \N$,  then 
\la      \label{Eq:P_le_s-cx_Q}
   P &\le^{}_{s\text{-cx}}   & Q   
\al
is defined to mean that $P,Q\in\Prob_{s-1}(\R)$ and that $Pf \le Qf$ holds for every $s$-convex 
function $f:\R\rightarrow\R$ such that $Pf$ and $Qf$ are well-defined (possibly infinite).
Thus $\le^{}_{1\text{-cx}}$ is just the usual stochastic order $\le^{}_{\text{st}}$ on $\Prob(\R)$, 
$\le^{}_{2\text{-cx}}$ is the usual convex order   $\le^{}_{\text{cx}}$ on $\Prob_1(\R)$, 
and $\le^{}_{3\text{-cx}}$ is what we use below. By considering the $s$-convex function $\pm(\cdot)^k$ 
with $k\in \{1,\ldots,s-1\}$, it is clear that~\eqref{Eq:P_le_s-cx_Q} necessitates
\la                                                  \label{Eq:First_s-1_moments_equal}
  \mu_j(P) &=& \mu_j(Q) \,\, \in \,\ \R  \quad\text{ for }j\in \{1,\ldots,s-1\}.
\al

For $x\in\R$ and $\alpha\in[0,\infty[$, we agree to the standard notation 
$x_-^\alpha \coloneqq (x_-)^\alpha$ and  $x_+^\alpha\coloneqq (x_+)^\alpha$
if $\alpha>0$,  and $x_-^0\coloneqq (x\le 0)$ and  $x_+^0\coloneqq (x\ge 0)$, which is not in general 
the same as  $(x_-)^0$ and $(x_+)^0$ due to $0^0\coloneqq 1$. 
 For a law $P \in\Prob(\R)$, let $F$ and $\overline{F}$ denote its ordinary and ``upper'' 
distribution functions, 
that is, $F(x) \coloneqq P(\mathopen]-\infty,x\mathclose])$ and $\overline{F}(x) \coloneqq P([x,\infty[)$ for $x\in\R$, and we then define
$F_k(t)$ and $\overline{F}_k(t)$ for  $k\in\N$ and  $t\in\R$ inductively by $F_1\coloneqq F$, 
$\overline{F}_1\coloneqq \overline{F}$,  
\la                           \label{Eq:Fbar_k_recursion}
 F_{k+1}(t) \,\ \coloneqq \,\ \int_{-\infty}^t F_k(x)\dd x,
  && \overline{F}_{k+1}(t) \,\ \coloneqq \,\ \int_t^\infty \overline{F}_{k}(x) \dd x,
\al
and hence get, as follows by inserting the right hand sides from~\eqref{Eq:Fbar_k_integral} into  
the integrals in~\eqref{Eq:Fbar_k_recursion}  and using Fubini, 
\la                                        \label{Eq:Fbar_k_integral}
 F_{k}(t) \,\ = \,\ \int\frac{(x-t)_-^{k-1}}{(k-1)!}\dd P(x), 
  &&  \overline{F}_{k}(t) \,\ = \,\ \int\frac{(x-t)_+^{k-1}}{(k-1)!}\dd P(x).
\al 
By~\eqref{Eq:Fbar_k_integral},  the functions $F_k$ and  $\overline{F}_k$ are  finite-valued
in particular if $P\in\Prob_{k-1}(\R)$, and then~\eqref{Eq:Fbar_k_recursion}
with $k-1$ in place of $k$ yields 
\la                                                          \label{Eq:Fbar_k_limits}
  \lim_{t\rightarrow-\infty} F_{k}(t)=0, &&   \lim_{t\rightarrow\infty} \overline{F}_{k}(t)=0. 
\al

In Theorem~\ref{Thm:CrossingPoints}(a,d) below, symmetry of $P-Q$ is to be understood in the usual sense of 
$(P-Q)(B)=(P-Q)(-B)$ for every Borel set $B\subseteq\R$.
 
\begin{Thm}[$\zeta$-distances, $s$-convex orderings, cut conditions]              \label{Thm:CrossingPoints}
Let $s\in \N$ and let $P,Q\in\Prob_{s-1}(\R)$ satisfy the moment condition~\eqref{Eq:First_s-1_moments_equal}. 
Let further  $F,\overline{F},G,\overline{G}$ denote the respective ordinary and complementary distribution 
functions of $P,Q$ and, with $F_k, \overline{F}_k , G_k, \overline{G}_k$ as in~\eqref{Eq:Fbar_k_recursion}  
and~\eqref{Eq:Fbar_k_integral}, let $H_k \coloneqq G_k - F_k$ and $\overline{H}_k \coloneqq \overline{G}_k - \overline{F}_k$
for $k\in\{1,\ldots,s\}$. 

\smallskip{\rm\textbf{(a)}} For $k\in\{1,\ldots,s\}$ and $t\in\R$, we have 
\la                                                      
   (-1)^{k-1}H_k(t) + \overline{H}_k(t+) &=& 0,      \label{Eq:H_k_vs_overline{H}_k}      \\
    (-1)^{k-1}H_k(t-) + \overline{H}_k(t) &=& 0,       \label{Eq:H_k(t-)_vs_overline{H}_k(t)}
\al
and, if $P-Q$ is symmetric, then also
\la                                                               \label{Eq:overline{H}_k_odd_or_even}
   \overline{H}_k(-t) &=& (-1)^k  \overline{H}_k(t+)\,;   
\al
here the one-sided limit signs,  namely ``$+$'' in the argument of $\overline{H}_k$ 
in \eqref{Eq:H_k_vs_overline{H}_k} and~\eqref{Eq:overline{H}_k_odd_or_even},  
and~``$-$'' in the argument of ${H}_k$ in~\eqref{Eq:H_k(t-)_vs_overline{H}_k(t)}, can be omitted if $k\ge 2$.

Let  $I$ denote the smallest  interval satisfying $P(I)=Q(I)=1$.
Then, for each $k\in\{1,\ldots,s\}$, we have $\overline{H}_k=0$ on $\R\setminus I $ and
\la                                              \label{Eq:H_k-limits}
 \lim_{t\rightarrow-\infty}\overline{H}_k(t) &=& \lim_{t\rightarrow\infty}\overline{H}_k(t)\,\ =\,\ 0.
\al
If in addition $P,Q\in \Prob_{s}(\R)$, then  we have
\la                                              \label{Eq:zeta_as_integral}
    \zeta_s(P,Q) &=& \int |\overline{H}_s(x)|\dd x,
\al                                    
and a function $f\in\cF_s$ satisfies  
\la                                                \label{Eq:zeta_s_attained} 
    \zeta_s(P,Q) &=& Qf -Pf 
\al
iff its Lebesgue-a.e.~existing derivative of order $s$  satisfies 
\la                                                \label{Eq:Char_f_attains_zeta_s}        
  f^{(s)}(x) &=& \begin{Bmatrix} -1\\ 1 \end{Bmatrix} 
  \,\ \text{ if }\,\  \overline{H}_s(x)\,\ \begin{Bmatrix} < \\ > \end{Bmatrix} \,\ 0,   
  \quad\text{ for Lebesgue-a.e.~$x\in I$}.
\al

\smallskip{\rm\textbf{(b)}} For $k\in\{1,\ldots,s\}$, let $(B_k)$ denote the condition ``$\overline{H}_k$ has 
at most $s-k$ sign changes and is lastly positive''. Then we have 
the implications $(B_1) \Rightarrow (B_2) \Rightarrow \ldots \Rightarrow (B_s)
\Leftrightarrow \overline{H}_s \ge 0 \Leftrightarrow P \le^{}_{s-\mathrm{cx}}Q$.  
If in addition $P,Q\in\Prob_s(\R)$,  then $P \le^{}_{s-\mathrm{cx}}Q$ is further equivalent to   
$\zeta_s(P,Q) = \frac1{s!}(\mu_s(Q)-\mu_s(P))$, that is, to \eqref{Eq:zeta_s_attained}
holding for  the function $f\in\cF_s$ given by 
\la            \label{Eq:Cube_function}
   f(x) &\coloneqq& \tfrac1{s!} x^s \quad\text{ for }x\in\R.
\al

\smallskip{\rm\textbf{(c)}} For $k\in\{1,\ldots,s\}$, let $(C_k)$ denote the condition 
``$\overline{H}_k$ has exactly $s-k+1$ sign changes and is lastly positive''. Then we have 
the implications  $(C_1) \Rightarrow (C_2) \Rightarrow \ldots \Rightarrow (C_s)$.
If in addition $P,Q\in\Prob_s(\R)$,  then $(C_s)$ is further equivalent to~\eqref{Eq:zeta_s_attained}
holding,  with some sign change point $x_0$ of $\overline{H}_s$, for  the function $f\in\cF_s$ given by 
\la                                                   \label{Eq:|x-x_0|^}
  f(x) &\coloneqq& \tfrac1{s!}|x-x_0|^s \quad \text{ for }x\in\R,
\al
and this remains true if ``some'' is replaced by ``some and every''. 
Further, if $(C_k)$ holds for some $k\in\{1,\ldots,s-1\}$, then each sign change point of  $\overline{H}_s$
belongs to the interior of the convex hull of the entries of every sign change tuple of  $\overline{H}_k$.

\smallskip{\rm\textbf{(d)}} Assume that we have $P,Q\in\Prob_s(\R)$, $P-Q$ symmetric, and $\overline{H}_s$ with 
exactly one sign change. Then $s$ is odd, and~\eqref{Eq:zeta_s_attained} holds with 
$f(x) \coloneqq|x|^s/s!$ for $x\in\R$.
\end{Thm}
\begin{proof}
(a) For every $t\in\R$,  \eqref{Eq:Fbar_k_integral} yields that 
\[
  (-1)^{k-1}F_k(t) + \overline{F}_k(t+) & =&  (-1)^{k-1}F_k(t-) + \overline{F}_k(t)  
  \,\ =\,\ \int \frac{(x-t)^{k-1}}{(k-1)!} \dd P(x)
\]
is a function of $\mu_1(P), \ldots,\mu_{k-1}(P)$, and $(-1)^{k-1}G_k(t) + \overline{G}_k(t+)$
is the same function of $\mu_1(Q), \ldots,\mu_{k-1}(Q)$; hence~\eqref{Eq:First_s-1_moments_equal}
yields~\eqref{Eq:H_k_vs_overline{H}_k} and~\eqref{Eq:H_k(t-)_vs_overline{H}_k(t)}. 
If now $P-Q$ is assumed to be symmetric, then, using this at the second step below, 
and using~\eqref{Eq:Fbar_k_integral} applied to $Q$ and to $P$ at the first and fourth steps,
and~\eqref{Eq:H_k_vs_overline{H}_k} at the fifth, 
we get~\eqref{Eq:overline{H}_k_odd_or_even} through
\[
 \overline{H}_k(-t)&=& \int\frac{(x+t)_+^{k-1}}{(k-1)!}\dd(Q-P)(x)
  \,\ =\,\ \int\frac{(-x+t)_+^{k-1}}{(k-1)!}\dd(Q-P)(x) \\
  &=& \int\frac{(x-t)_-^{k-1}}{(k-1)!}\dd(Q-P)(x) \,\ =\,\ H_k(t) \,\ =\,\ (-1)^k\overline{H}_k(t+).
\]

Back in the general case, since $(\cdot-t)_+^{k-1}$ is $(P+Q)$-a.e.~equal to 
a polynomial of degree $\le$ $k-1$ if $t\in \R\setminus I$, 
namely  $(P+Q)$-a.e.~$(\cdot-t)_+^{k-1}= (\cdot-t)_{}^{k-1}$ if $\{t\}<I$ and  $(\cdot-t)_+^{k-1}= 0$ if $\{t\}>I$,
we get $\overline{H}_k=0$ on $\R\setminus I$.
Claim \eqref{Eq:H_k-limits} follows using \eqref{Eq:Fbar_k_limits} and \eqref{Eq:H_k_vs_overline{H}_k}.

Assume now $P,Q\in\Prob_s(\R)$.  If $f\in\cF_s$, then the representation $f(x)=\sum_{j=0}^{s-1} \frac{f^{(j)}(0)}{j!}x^j+
\int_0^x\frac{(x - y )^{s-1}}{(s-1)!}f^{(s)}(y)\dd y
 = \sum_{j=0}^{s-1} \frac{f^{(j)}(0)}{j!}x^j+ \int_\R \big((0\le y<x)-(x\le y<0)   \big)
\frac{(x - y )^{s-1}}{(s-1)!}f^{(s)}(y)\dd y$
and a Fubini calculation, valid due to $\|f^{(s)}\|^{}_\infty \le 1$  and the moment assumption 
just introduced, and using~\eqref{Eq:H_k_vs_overline{H}_k} with $k=s$, yield the  formula 
\la                                         \label{Eq:Qf-Pf_via_H_s}
 Qf -Pf  &=& \int f^{(s)}(x) \overline{H}_s(x)\dd x.
\al
By applying~\eqref{Eq:Qf-Pf_via_H_s} to $f\in \cF_s^\infty$ and using $\|f^{(s)}\|^{}_\infty \le 1$ we get 
``$\le$'' in~\eqref{Eq:zeta_as_integral}. By applying~\eqref{Eq:Qf-Pf_via_H_s} to a function $f\in \cF_s$  
with $f^{(s)}(x)= \sign( \overline{H}_s(x))$ for Lebesgue-a.e.~$x$, and using Theorem~\ref{Thm:First_facts_on_zeta_distances}(d),
we get ``$\ge$'' in~\eqref{Eq:zeta_as_integral}. Finally, \eqref{Eq:zeta_as_integral}  and~\eqref{Eq:Qf-Pf_via_H_s} yield
the claim involving~\eqref{Eq:Char_f_attains_zeta_s}.

(b) Using~\eqref{Eq:H_k-limits}, the implications  $(B_1) \Rightarrow (B_2) \Rightarrow \ldots \Rightarrow (B_s)$ 
follow from Lemma~\ref{Lem:Ineq_S(f)_S(f')} up to the statement involving~\eqref{Eq:Ineq_S(f)_S(f')}, 
since~\eqref{Eq:Fbar_k_recursion} yields $\overline{H}_{k+1}'(t) = - \overline{H}_{k}(t)$ for 
$k\in \{1,\ldots,s-1\}$ and $t\in\R$, except for at most countably many $t$ in case of $k=1$. 
The equivalence  $(B_s)\Leftrightarrow \overline{H}_s\ge 0$ is trivial, 
and the equivalence $\overline{H}_s \ge 0 \Leftrightarrow P \le^{}_{s-\mathrm{cx}}Q$
is \cite[Theorem 3.2]{DenuitLefevreShaked1998}, using \eqref{Eq:Fbar_k_integral}.
Since~\eqref{Eq:Char_f_attains_zeta_s} holds for $f$ from~\eqref{Eq:Cube_function} 
iff $\overline{H}_s\ge 0$, using the left-continuity of~$\overline{H}_s$ and also $\overline{H}_s= 0$
on $\R\setminus I$ for the ``only if'' part, 
the final equivalence follows from part~(a). 

(c) Let $k\in\{1,\ldots,s-1\}$ and assume $(C_k)$. Then, as in the proof of part~(b),
we deduce that $\overline{H}_{k+1}$ has at most $s-k$ sign changes and is lastly positive.
If $\overline{H}_{k+1}$ even had at most $s-k-2 =(s-1)-(k+1)\in\N_0$ sign changes, then $k+1\le s-1$, 
and hence part~(b) applied  with $s-1$ in place of $s$ would yield $P \le^{}_{(s-1)-\mathrm{cx}}Q$
and hence $\zeta_{s-1}(P,Q)=\frac{1}{(s-1)!}(\mu_{s-1}(Q)-\mu_{s-1}(P))=0$
and thus $P=Q$ by Theorem~\ref{Thm:First_facts_on_zeta_distances}(b), in contradiction to $(C_k)$.
If $\overline{H}_{k+1}$ had exactly $s-k-1$ sign changes, then, on the one hand, part~(b)
as it stands would yield $\overline H_s\ge0$, but on the other hand, 
by~$(C_k)$, there would exist a $t_0\in\R$ such that the left-continuous function $(-1)^{s-k+1}\overline{H}_k$ 
would be  $\ge 0$ on $\mathopen]-\infty,t_0\mathclose]$ and actually $>0$ on some nondegenerate subinterval  $]t_1,t_0]$,
so that, in view of $\overline H_k(t+)=(-1)^kH_k(t)$ by~\eqref{Eq:H_k_vs_overline{H}_k}, the expression  
$(-1)^{s+1} H_k(t) =(-1)^{s-k+1}\overline H_k(t+)$ would 
be $\ge 0$ for $t\in\mathopen]-\infty,t_0\mathclose[$  and $>0$ for $t\in\mathopen[t_1,t_0\mathclose[$, and hence
$\overline{H}_s(t_0) = (-1)^{s}H_s(t_0-) <0$ by~\eqref{Eq:H_k(t-)_vs_overline{H}_k(t)}
and the recursion \eqref{Eq:Fbar_k_recursion},  a contradiction. Thus indeed $(C_{k+1})$ holds.

Let $x_0\in\R$ and $f$ be as in \eqref{Eq:|x-x_0|^}. Then $f^{(s)}(x)=\sign(x-x_0)$ for $x\in\R\setminus\{x_0\}$,
and hence~\eqref{Eq:Char_f_attains_zeta_s} holds iff $(x_0)$ is a sign change tuple for $\overline{H}_s$
and $\overline{H}_s$ is lastly positive. Hence the stated equivalence involving ``some'' and ``some and every''
follows using part~(a). 
  
The final claim of part (c) follows using the ``More precisely'' statement of Lemma~\ref{Lem:Ineq_S(f)_S(f')}.

(d) Suppose that $0$ were no sign change point of $\overline{H}_s\eqqcolon h$.
Then at least one of the following three conditions would be violated: 
(i)~$h(x)h(y)\ge 0$ for $x,y\in\mathopen]-\infty,0\mathclose[$,
(ii)~$h(x)h(y)\ge 0$ for $x,y\in\mathopen]0,\infty\mathclose[$,
(iii)~$h(x)h(y)<0$ for some $x<0<y$. If (i) or (ii) were false, that is, $h(x)h(y)<0$ for some $x,y\in I$
with $I=\mathopen]-\infty,0\mathclose[$ or  $I=\mathopen]0,\infty\mathclose[$, 
then~\eqref{Eq:overline{H}_k_odd_or_even} would yield $h(-x+)h(-y+) <0$,
and hence $h(u)h(v)<0$ for some $u,v\in -I$, leading to $S^{-}(h)\ge 2$, a contradiction. 
If~(i) and~(ii) were true but (iii) not, then $S^{-}(h)=0$, again  a contradiction. 

Thus $0$ is a sign change point of $\overline{H}_s$, and hence part~(c) yields, since condition $(C_s)$ is fulfilled,
that~\eqref{Eq:zeta_s_attained} holds with $f$ from~\eqref{Eq:|x-x_0|^} with $x_0=0$.  

Hence, if $s$ were even, then~\eqref{Eq:zeta_s_attained} would hold with $f$ from~\eqref{Eq:Cube_function}, 
but then by part~(b) we would have $(B_s)$, that is, $\overline{H}_s$ would have no sign changes, a contradiction.
Therefore $s$ is odd.
\end{proof}
 
From  the following example, which in particular computes $\epsilon_1$ from~\eqref{Eq:Def_epsilon_n},   
the results~\eqref{Eq:zeta_3(Q,N)} and~\eqref{Eq:zeta_4(Q,N)}  are used in the proofs of Theorems~\ref{Thm:Main(noniid)} 
and~\ref{Thm:epsilon_n} in section~\ref{Sec:zeta3-between-N-B} below. 
       
\begin{Example}                                             \label{Example:zeta_3(Q,N)_zeta_4(Q,N)}
{\em Let $Q\coloneqq \frac12(\delta_{-1}+\delta_1)$. Then we have 
\la
  \epsilon_1 \,\ =\,\ \zeta_3(Q,\mathrm{N}) &=& \frac16\bigg( \frac{4}{\sqrt{2\pi}} 
      -1\bigg) 
       \,\ < \,\ 0.0993,   \label{Eq:zeta_3(Q,N)}\\
   \zeta_4(Q,\mathrm{N}) &=& \frac1{12} \,\ <\,\ 0.0834, \label{Eq:zeta_4(Q,N)} \\
   \zeta_s(Q,\mathrm{N}) &=& \infty \quad\text{ for } s\in\mathopen]4,\infty\mathclose[. \label{Eq:zeta_>4(Q,N)}
\al}%
\begin{proof} Claim~\eqref{Eq:zeta_>4(Q,N)} follows from Theorem~\ref{Thm:First_facts_on_zeta_distances}(c)
with $m\ge 4$, since $\mu_4(Q)=1\neq 3 =\mu_4(\mathrm{N})$. 

For proving \eqref{Eq:zeta_3(Q,N)} and~\eqref{Eq:zeta_4(Q,N)} using Theorem~\ref{Thm:CrossingPoints}, let us 
change here the notation and put for the rest of this proof
\[
    P &\coloneqq& \tfrac12(\delta_{-1}+\delta_1),\qquad Q\coloneqq \mathrm{N}.
\] 
Then, using from now on the notation of Theorem~\ref{Thm:CrossingPoints} with these $P,Q$, 
and first with $s\in\{1,2,3,4\}$ arbitrary, we have~\eqref{Eq:First_s-1_moments_equal}, 
and the function $\overline{H}_1=\overline{G}-\overline{F}$
obviously has the unique sign change tuple $(-1,0,1)$ and hence exactly three sign changes, and is lastly positive.

If now $s=4$, then assumption $(B_1)$ of  Theorem~\ref{Thm:CrossingPoints}(b) is fulfilled,
and, with $f(x)\coloneqq x^4/4!$ from~\eqref{Eq:Cube_function}, we  accordingly get
\[
  \zeta_4(P,Q) &=& Qf-Pf \,\ = \,\  \frac1{4!}(3-1) \,\ =\,\ \frac1{12}\,.
\] 

If, finally, $s=3$, then assumption $(C_1)$ of  Theorem~\ref{Thm:CrossingPoints}(c) is fulfilled,
hence so is $(C_3)$, and,  by symmetry of $P$ and of $Q$,   Theorem~\ref{Thm:CrossingPoints}(d) now yields  
\[
  \zeta_3(P,Q) &=& \frac1{3!} \left( Q|\cdot|^3 - P|\cdot|^3\right) \,\ = \,\  
   \frac16\bigg( \frac{4}{\sqrt{2\pi}} 
    -1\bigg).
\]
\end{proof}
\end{Example}

\section{Proof of the main result}                \label{Sec:Main_proofs}
\begin{proof}[Proof of Theorem~\ref{Thm:Main}]

We will use random variable notation whenever this appears to be more convenient. So, in addition to the 
assumptions of Theorem~\ref{Thm:Main}, let $X_i\sim P_i$ and $Y_i\sim Q_i$ be $2n$ independent random variables
on some probability space with expectation operator $\E$. Without loss of generality, we assume the $P_i$ to be 
centred, that is,  $\E X_i =0$ for each~$i$.

\medskip
\textit{Step 1.} Equality in \eqref{Eq:Main_inequality_non-i.i.d.} occurs under the stated
conditions. Indeed, we then have  $ \widetilde{\bigconv_{i=1}^n Q_i}\, f =0$ by symmetry,
and thus
\[
 \text{L.H.S.\eqref{Eq:Main_inequality_non-i.i.d.}}
  &=& \left| \E\, c\left( \frac1{\sigma}\sum_{i=1}^n X_i \right)^3  \right|
  \,\ = \,\   \frac1{6\sigma^3}  \left| \sum_{i=1}^n  \E X_i^3 \right|     6 |c| \\
  &=&  \frac1{6\sigma^3} \sum_{i=1}^n  \sigma_i^3
    \left| \E  \left(\frac{X_i}{\sigma_i} \right)^3  \right|  \| f''\|^{}_{\mathrm{L}}
 \,\ = \,\ \text{R.H.S.\eqref{Eq:Main_inequality_non-i.i.d.}}
\]
by  using in the third step above the additivity of the third centred moment
for independent random variables, that is,  \eqref{Eq:Cumulants_additive}  with $\ell=3$, and in last step the equality statement in Example~\ref{Example:Thm_6_of_Irina_2014}, that is, a rather easy part of~\cite[Theorem~6]{Shevtsova2014JMAA}.

\medskip
\textit{Step 2.} We may assume that the Banach space $E$ is the real line $\R$,
with the norm being the usual modulus.   Indeed, assume Theorem~\ref{Thm:Main}
to be true in this special case. Then, for the given general $f$,
the Hahn-Banach theorem \cite[Theorem 5.20]{Rudin}
yields an $\R$-linear functional $\ell : E \rightarrow \R $   of norm $1$
satisfying  the first of the following equalities
\[
 \text{L.H.S.\eqref{Eq:Main_inequality_non-i.i.d.}}
  &=& \ell\left(\widetilde{\bigconv_{i=1}^n P_i}\ f-\widetilde{\bigconv_{i=1}^n Q_i}\ f\right)
  \,\ =\,\ \widetilde{\bigconv_{i=1}^n P_i} \ \ell\!\circ\! f
           -\widetilde{\bigconv_{i=1}^n Q_i}\ \ell\!\circ\! f,
\]
and thus an application of inequality~\eqref{Eq:Main_inequality_non-i.i.d.} to
$\ell\!\circ\! f$
in place of $f$ and using
$\| (\ell\!\circ\! f)'' \|^{}_{\mathrm{L}}
=  \|  \ell\!\circ\! f'' \|^{}_{\mathrm{L}} \le \|f'' \|^{}_{\mathrm{L}}$
yields inequality~\eqref{Eq:Main_inequality_non-i.i.d.} as stated (for example, in the particular case of $E=\C$ we may put 
$\ell(z):=\Re(cz),$ where $\Re$ stands for the real part and $c=c_f\in\C$ is such that $|c|=1$ and
$c\cdot\big(\widetilde{\bigconv_{i=1}^n P_i}\ f-\widetilde{\bigconv_{i=1}^n Q_i}\ f\big)$ is real and $\geq0$).

\medskip
\textit{Step 3.} It is enough to prove inequality~\eqref{Eq:Main_ineq_rewritten},
since we have $| P f -Qf | \le  \|f''\|^{}_{\mathrm{L}}\ \zeta_3(P,Q)$ for $P,Q \in\Prob_3(\R)$ 
and $f\in \cC^{2,1}(\R,\R)$  by~\eqref{Eq:Pf-Qf_vs_zeta(P,Q)} with $s\coloneqq3$, and in view of Steps~1 and~2.

\medskip
\textit{Step 4.}  It is enough to prove inequality~\eqref{Eq:Main_ineq_rewritten} in case
of $n=1$, since assuming this special case to be true yields the penultimate step below in
\[
 \text{L.H.S.\eqref{Eq:Main_ineq_rewritten}}
   &=&  \zeta_3\left(\frac1\sigma \sum_{i=1}^n X_i \, , \, \frac1\sigma \sum_{i=1}^n Y_i  \right)
  \,\ = \,\  \frac1{\sigma^3} \zeta_3\left(\sum_{i=1}^n X_i \, , \, \sum_{i=1}^n Y_i  \right)
  \,\ \le  \,\ \frac1{\sigma^3}  \sum_{i=1}^n  \zeta_3 ( X_i  ,  Y_i)\\
 & = &    \frac1{\sigma^3}  \sum_{i=1}^n  \sigma_i^3\zeta_3 ( \widetilde{X_i}  ,  \widetilde{Y_i} )
      \,\ \le \,\ \frac1{\sigma^3}  \sum_{i=1}^n \sigma_i^3  \frac{\beta_i A(\rho_i)  }{6 \sigma_i^3}
   \,\ =\,\ \text{R.H.S.\eqref{Eq:Main_ineq_rewritten}}, 
\]
where we have used the homogeneity~\eqref{Eq:homogeneity_of_zeta} at the second and fourth steps,  
and the semiadditivity~\eqref{Eq:semiadditivity_of_zeta}  at the third.

\medskip
\textit{Step 5.} Let us write for the rest of this proof
\la                                                                   \label{Eq:Def_Q_symm_two-point}
  Q &\coloneqq& \tfrac12(\delta_{-1}+\delta_1).
\al
By Step 4, it remains to prove that we have
\la                    \label{Eq:zeta_3(P,Q)-...le_0}
  \zeta_3(P,Q)  - \frac{B(\rho(P))}6 &\le& 0 
\al
for $P\in \widetilde{\cP_3}$ or, equivalently in view of the alternative representation~\eqref{Eq:zeta_s_in_Prob_s}
of $\zeta_3$, that
\la                                   \label{Eq:Pf-Qf-B(rho(p))/6_le_0}
  Pf -Qf  - \frac{B(\rho(P))}6 &\le& 0
\al
holds for $P\in\widetilde{\cP_3}$ and  $f\in\cC^{2,1}(\R,\R)$ with $\|f''\|_{\mathrm{L}}\le1$. Let
$f_1(x)\coloneqq x$, $f_2(x)\coloneqq x^2$, and $f_3(x) \coloneqq |x|^3$ for $x\in\R$.
Given now $P\in\widetilde{\cP_3}$ and $f\in\cC^{2,1}(\R,\R)$ with $\|f''\|_{\mathrm{L}}\le1$, we can apply Theorem~\ref{Thm:Richter}(a)
to $P$ and to the functions $f_1,f_2,f_3$, and $f_4 \coloneqq f$ to conclude,
since the left hand side of~\eqref{Eq:Pf-Qf-B(rho(p))/6_le_0} is a function of $Pf_3$ and $Pf_4$,
that  it is enough to prove~\eqref{Eq:Pf-Qf-B(rho(p))/6_le_0} under the additional assumption
that $P$ has at most $5$ support points. (Using instead of Theorem~\ref{Thm:Richter}(a) the a bit deeper
Theorem~\ref{Thm:Richter}(b), which applies by the continuity of the functions $f_i$ and the connectedness
of $\R$, we could reduce ``$5$'' above to ``$4$'', but this does not appear to help in what follows.)
Hence it is enough to prove~\eqref{Eq:zeta_3(P,Q)-...le_0} for $P\in\cP$ where
\[
  \cP &\coloneqq& \{ P \in \Prob(\R) : \#\mathrm{supp\,}P<\infty, Pf_1 =0, Pf_2 =1\}.
\]
Let $F(P)$ be the left hand side of~\eqref{Eq:zeta_3(P,Q)-...le_0}  for $P\in\cP$.
Then $F$ is a convex $\R$-valued functional on $\cP$, since $P\mapsto \rho(P) = P f_3$ is linear on $\cP$,
$B$ is concave by Lemma~\ref{Lem:A_and_B}, and $P\mapsto \zeta_3(P,Q)$ is convex since it is the supremum
of the affine functionals $P\mapsto Pf-Qf$ with $f\in\cC^{2,1}(\R,\R)$.
Hence Tyurin's Theorem~\ref{Thm:Tyurin_reduction}, with $k\coloneqq2$, shows that it is enough
to prove~\eqref{Eq:zeta_3(P,Q)-...le_0}  for $P$ standardized and having at most three support points.
So, for the remaining two steps, let 
\la                                                                 \label{Eq:Def_P_three-point}
 P &=& p\delta_\alpha + q\delta_\beta + (1-p-q)\delta_\gamma
\al
with some $\alpha\le \beta \le \gamma$, $p,q > 0$, $p+q<1$, $p\alpha+q\beta+(1-p-q)\gamma=0$, and $p\alpha^2+q\beta^2+(1-p-q)\gamma^2=1$. 
Let us further apply the notation $\overline{H}_k$ of Theorem~\ref{Thm:CrossingPoints} 
with $s\coloneqq 3$  to the present $P$ from \eqref{Eq:Def_P_three-point} and $Q$ 
from~\eqref{Eq:Def_Q_symm_two-point}. 
Then $\overline{H}_1$ has at most $5-2=3$ sign changes, since with $S\coloneqq \{\alpha,-1,\beta,1,\gamma\}$, 
only the elements of $S\setminus\{\min S,\max S\}$ can be sign changes.

\medskip
\textit{Step 6.} Assume in this step that $\overline{H}_1$ has at most two sign changes. Then, 
since $\overline{H}_1$ or $-\overline{H}_1$ is lastly positive, Theorem~\ref{Thm:CrossingPoints}(b) 
applied to $(P,Q)$ or to $(Q,P)$ yields the first equality in 
\[
  \zeta_3(P,Q) &=& \left|\int \frac{x^3}{6} \dd(P-Q)(x) \right|
    \,\ = \,\  \frac{1}{6}\left|\int x^3\dd P(x)  \right| \,\ \le\,\ \frac{B(\varrho(P))}{6}, 
\]
where the final inequality comes from~\cite[Theorem~6]{Shevtsova2014JMAA}, that is,
from~\eqref{Eq:Thm_6_of_Irina_2014} of Example~\ref{Example:Thm_6_of_Irina_2014}.

\smallskip
\textit{Step 7.} Assume finally that $\overline{H}_1$ has exactly three sign changes.
Then we have $\alpha<-1<\beta<1<\gamma$,  and the (unique) sign change tuple of $\overline{H}_1$ is $(-1,\beta,1)$, 
with the interior of the convex hull of its coordinates being $\mathopen]-1,1\mathclose[$.
Hence Theorem~\ref{Thm:CrossingPoints}(c), with $s=3$ and with the condition $(C_1)$ being fulfilled, 
yields the existence of an $r\in \mathopen]-1,1\mathclose[$ satisfying 
\la                                                               \label{Eq:zeta_3(P,Q)_explicit}
  \zeta_3(P,Q) &=& \frac16\left( \int|x-r|^3\dd P(x)-  \int|x-r|^3\dd Q(x)       \right).
\al 
If $r=0$, then $\text{R.H.S.\eqref{Eq:zeta_3(P,Q)_explicit}} = \frac16(\rho(P) -1) \le  \frac16B(\varrho(P))$, 
using Lemma~\ref{Lem:A_and_B}. 

So let now $r\neq 0$. Then there is a (unique) two-point law $P'\in \widetilde{\cP_3}$ 
with $\varrho(P') = \varrho(P)$ and concentrated in points $v\cdot\sgn(r)$ and $ -u\cdot\sgn(r)$ 
with certain $u>v>0$, compare the distribution of $X_\varrho$ in subsection~\ref{Subs:Intro} above.
Lemma~\ref{Lem:E|X-t|^3<=E(a+bX+cX^2+d|X|^3)} yields
\[
  \int |x-r|^3\dd P(x) &<& a_r(u,v)+c_r(u,v) + d_r(u,v)\varrho(P) \,\ =\,\   \int|x-r|^3\dd P'(x)
\]
using also standardizedness of $P,P'$ and $\varrho(P') = \varrho(P)$.
Hence, using also~\eqref{Eq:zeta_3(P,Q)_explicit} in the first step below, we get 
\[
 \zeta_3(P,Q) &<& \int\frac16|x-r|^3\dd(P'-Q)(x) \,\ \le \,\ \zeta_3(P',Q).
\]
Finally,  Step~6 applied to $P'$ in place of $P$, which is legitimate since the  
$\overline{H}_1$ corresponding to the two-point law $P'$ has at most two sign changes, yields 
$\zeta_3(P',Q) \le\frac16 B(\varrho(P')) = \frac16B(\varrho(P))$.
\end{proof}

\section{Proofs involving $\zeta_3$-distances
between normal and convolutions of symmetric two-point laws}              \label{Sec:zeta3-between-N-B}
\begin{proof}[Proof of Theorem~\ref{Thm:Main(noniid)}]        
Inequality~\eqref{Eq:Main_normal(noniid)} follows from~\eqref{Eq:Main_ineq_rewritten} 
in Theorem~\ref{Rem:Main_thm_rewritten} by using the triangle inequality for $\zeta_3$ 
recalled in Theorem~\ref{Thm:First_facts_on_zeta_distances}(b). 
For the remaining claim, we assume without loss of generality that 
\la                 \label{Eq:sum_sigma_i^2=1}
  \sum_{i=1}^n\sigma_i^2&=&1.
\al
Let $Y,Y_1,\ldots,Y_n,Z,Z_1,\ldots,Z_n$ be independent r.v.'s with $Y\sim\frac12(\delta_{-1}+\delta_{1})$,
$Y_i\sim Q_i$ and hence $Y_i\sim \sigma_i Y$, $Z\sim \mathrm{N}$, and $Z_i\sim \mathrm{N}_{\sigma_i}$ 
and hence $Z_i\sim\sigma_i Z$, for $i\in\{1,\ldots,n\}$.
Let further  $T_k\coloneqq Z_1+\ldots+Z_k+Y_{k+1}+\ldots+Y_n$ for $k\in\{0,\ldots,n\}$. 
Then, using~\eqref{Eq:sum_sigma_i^2=1}, we get $T_0 \sim  \widetilde{\bigconv_{i=1}^nQ_i} =  \bigconv_{i=1}^nQ_i$ and 
$T_n \sim\mathrm{N}$ and hence, writing in this proof $\epsilon_n$ for a quantity more general  than 
the one introduced in~\eqref{Eq:Def_epsilon_n},  we get 
\[
 \eps_n &\coloneqq & \zeta_3\bigg( \widetilde{\bigconv_{i=1}^nQ_i},\,\mathrm{N} \bigg) 
  \,\ = \,\ \zeta_3(T_0,T_n) \,\ \le\,\  \zeta_3(T_0,T_1)+\sum_{k=1}^{n-1}\zeta_3(T_k,T_{k+1})
\]
by using the triangle inequality at the last step.

The regularity~\eqref{Eq:regularity_of_zeta} and the homogeneity~\eqref{Eq:homogeneity_of_zeta}
of $\zeta_3$ yield
$$
\zeta_3(T_0,T_1)\,\ \le\,\ \zeta_3(Y_1,Z_1)\,\ =\,\ \sigma_1^3\zeta_3(Y,Z). 
$$
Noting that the r.v.\ $Z_1+\ldots+Z_k$ occurring in $T_k$ and in  $T_{k+1}$ 
has the centred normal distribution with variance $\sum_{i=1}^k\sigma_i^2$   
and applying Lemma~\ref{LemZeta_s<=Zeta_s+l} with $s=3$ and $t=1,$ we get
$$
\zeta_3(T_k,T_{k+1})\, \le\,   \sqrt{\frac2\pi}\cdot\frac{\zeta_4(Y_{k+1},Z_{k+1})}{\sqrt{\sum_{i=1}^k\sigma_i^2}}  
 \, =\,  \sqrt{\frac2\pi}\cdot \frac{\zeta_4(Y,Z)\sigma_{k+1}^4}{\sqrt{\sum_{i=1}^k\sigma_i^2}}
 \quad \text{ for } k\in\{1,\ldots,n\!-\!1\},
$$
so that 
$$
 \eps_n\,\ \le\,\  \zeta_3(Y,Z)\sigma_1^3+ \sqrt{\frac2\pi}\zeta_4(Y,Z) \sum_{k=1}^{n-1}\frac{\sigma_{k+1}^4}{\sqrt{\sum_{i=1}^k\sigma_i^2}}.
$$
Using now the assumptions $\sigma_1\ge\sigma_2\ge\ldots\ge\sigma_n$ and~\eqref{Eq:sum_sigma_i^2=1}, we  have $\sigma_1^2+\ldots+\sigma_k^2\ge k/n$ and also $\sigma_1^2+\ldots+\sigma_k^2\geq k\sigma_{k+1}^2$, which yields 
\la                                          \label{Eq:eps_n_general_upper_bound}
 \eps_n
&\le& \zeta_3(Y,Z)\sigma_1^3+ \sqrt{\frac2\pi}\zeta_4(Y,Z) \sum_{k=1}^{n-1}\frac{\sigma_{k+1}^3\min\{1,\sqrt{n}\sigma_{k+1}\}}{\sqrt{k}}.
\al
Inserting now the values for  $\zeta_3(Y,Z)$ and $\zeta_4(Y,Z$) from \eqref{Eq:zeta_3(Q,N)} 
and~\eqref{Eq:zeta_4(Q,N)}  in Example~\ref{Example:zeta_3(Q,N)_zeta_4(Q,N)} yields the claim.
\end{proof}

\begin{proof}[Proof of Theorem~\ref{Thm:epsilon_n}] \label{Page:proof-of-Thm:epsilon_n}
For the upper bound we observe that formula~\eqref{Eq:eps_n_general_upper_bound},
specialized to the homoscedastic case $\sigma_1=\ldots=\sigma_n=1/\sqrt{n}$,  yields 
\[
\eps_n &=& \zeta_3\left(\widetilde{\mathrm{B}_{n,\frac12}} , \mathrm{N}\right) \,\ \le\,\ \frac{\zeta_3(Y,Z)}{n^{3/2}}+ \sqrt{\frac2\pi}
\cdot\frac{\zeta_4(Y,Z)}{n^{3/2}}\sum_{k=1}^{n-1}\frac{1}{\sqrt{k}},
\]
which can further be simplified  by use of Lemma~\ref{Lem:zeta(1/2)} to give 
\[
  \eps_n & < & 2\sqrt{\frac2\pi} \cdot\frac{\zeta_4(Y,Z)}{n} + \frac{\zeta_3(Y,Z) + \zeta(\frac12)\sqrt{\frac2\pi}\zeta_4(Y,Z)}{n^{3/2}},
\]
and now the claimed upper bound for $\epsilon_n$ follows if we substitute the explicit values   
of $\zeta_3(Y,Z)$ and $\zeta_4(Y,Z)$ as in the preceding proof. 

For the lower bound, let us recall for  $n,k\in\N_0$ the $k$th Krawtchouk polynomial $P_k^n$ associated
to the symmetric binomial law $\mathrm{B}_{n,\frac12}$ as defined 
in~\cite[pp.~130, 151--154, the case of $q=2$ and hence $\gamma=1$]{MacWilliams_Sloane}
and also, with the unnecessary restriction $k\le n$,
in~\cite[section 6.2 on p.~298, the special case of $p=\frac12$
and hence $\gamma =1$]{DiaconisZabell}, that is,
\[
 P^n_k(x) &\coloneqq& \sum_{j=0}^k(-1)^j\binom{x}{j}\binom{n-x}{k-j}
 \quad\text{ for } x\in\R,
\]
so that we have in particular
\[
 &&P_0^n(x) \,\ = \,\ 1,\qquad P^n_1(x) \,\ = \,\ -2\left(x - \tfrac{n}2\right)
\]
and the recursion
\[
 (k+1)P^n_{k+1}(x)  &=& (n-2x)P_k^n(x) - (n-k+1)P^n_{k-1}(x)
\quad\text{ for }k\in\{1,\ldots,n-1\}
\]
and hence further
\[
  P_2^n(x) &=& 2 \left(\left(x-\tfrac{n}2\right)^2-\tfrac{n}4\right)\,,
     \\
 P_3^n(x) & =& -\tfrac43\left(x-\tfrac{n}{2} \right)^3
                +\left( n -\tfrac23\right)\left(x-\tfrac{n}2 \right).
\]
If now $n,k\in \N$, then, from the cited sources, we have for   $a\in\N_0$
\la                      \label{Eq:Closed_summation}
 \sum_{x=0}^a P^n_k(x)\mathrm{b}_{n,\frac12}(x)
  &=& \tfrac{n-a}k  P^{n-1}_{k-1}(a)\mathrm{b}_{n,\frac12}(a),
\al
and hence in particular
\[
  \sum_{x=0}^a P^n_1(x) \mathrm{b}_{n,\frac12}(x)
    &=& (n-a)\mathrm{b}_{n,\frac12}(a)  ,\\
 \sum_{x=0}^a   P^n_3(x)  \mathrm{b}_{n,\frac12}(a)
 &=& \tfrac23 (n-a)\left(\left( a- \tfrac{n-1}{2}\right)^2-\tfrac{n-1}4 \right)\mathrm{b}_{n,\frac12}(a)
\]
and thus
\[
\sum_{x=0}^a \left(x-\tfrac{n}2\right)\mathrm{b}_{n,\frac12}(x)
   &=& \tfrac{a-n}2 \mathrm{b}_{n,\frac12}(a),  \\
\sum_{x=0}^a \left(x-\tfrac{n}2\right)^3\mathrm{b}_{n,\frac12}(x)
  &=&   \sum_{x=0}^a \left( -\tfrac34 P_3^n(x) - \left(\tfrac3{8}n-\tfrac14\right) P_1^n(x)
    \right) \mathrm{b}_{n,\frac12}(x) \\
  &=& \tfrac{a-n}2
   \left( \left(a-\tfrac{n-1}{2}\right)^2 + \tfrac12 n -\tfrac14\right) \mathrm{b}_{n,\frac12}(a)   ,
\]
and finally
\la                                                                           \label{Eq:Cent_3rd_abs_mom_sym_bin}
 \sum_{x=0}^n\left|x-\tfrac{n}2\right|^3\mathrm{b}_{n,\frac12}(x)
   &=& -2\sum_{x=0}^{\lfloor\frac n2\rfloor}\left(x-\tfrac{n}2 \right)^3\mathrm{b}_{n,\frac12}(x)\\ \nonumber
   &=& \left\{ \begin{array}{ll}                                              
           \tfrac{1}{4}n^2\mathrm{b}_{n,\frac12}(\tfrac{n}{2}) &\text{ if $n$ is even}, \\
           \left( \tfrac{1}{4}n^2 +\tfrac1{8}n-\tfrac18 \right)
                   \mathrm{b}_{n,\frac12}(\lfloor\tfrac{n}{2}\rfloor)     &\text{ if $n$ is odd}.
        \end{array}\right.
\al
Recalling the local Edgeworth expansion for binomial laws (see, e.g.,~\cite[\S\,51, Theorem\,1]{GnedenkoKolmogorov1949})
\[
 \sqrt{\tfrac{n}4} \mathrm{b}^{}_{n,\frac12}(k)
   &=& \Phi'(z)  \left( 1 - \frac{z^4-6z^2+3}{12n} \right)    +O(n^{-2})
\]
uniformly in $z \coloneqq (k-\frac{n}2)/\sqrt{\frac{n}4}$ with $k\in\Z$, we thus get,
writing $\alpha_n \coloneqq \frac{n}{2} -  \lfloor \frac{n}{2}\rfloor$,
and using at the last step below $2\alpha_n^2 =\alpha_n$,
\[
 \epsilon^{}_n    &\ge& \left| \int \frac{|\cdot|^3}6
    \dd\left(\mathrm{N} - \widetilde{\mathrm{B}_{n,\frac12}} \right) \right|  \\
  &=& \frac16 \left| \frac{4}{\sqrt{2\pi}}   - \frac{2^3}{n^{3/2}}  \sum_{x=0}^n\left|x-\tfrac{n}2\right|^3\mathrm{b}_{n,\frac12}(x)
    \right| \\
  &=& \frac16 \left|\frac{4}{\sqrt{2\pi}}  -\frac{2^3}{n^{3/2}}  \text{R.H.S.\eqref{Eq:Cent_3rd_abs_mom_sym_bin}}   \right|  \\
 &=& \frac16\left| \frac{4}{\sqrt{2\pi}}  -\frac{2^3}{n^{3/2}}\left(\tfrac{n^2}4  +\alpha_n\tfrac{n}4 +O(1)\right)
   \sqrt{\tfrac{4}n} \left( \Phi'\left( \frac{-\alpha_n}{\sqrt{\tfrac{n}{4}}}   \right)
     \left(1 - \tfrac{ 3 + O(n^{-1})}{12n}\right) + O(n^{-2})  \right)
           \right| \\
 &=& \frac{4}{6\sqrt{2\pi}} \left| 1- \left(1+\tfrac{\alpha_n}{n}\right)
     \left( 1- \tfrac{ 2\alpha_n^2}{n}\right)\left( 1- \tfrac{1}{4n}\right)  + O(n^{-2})     \right| \\
 &=& \frac1{6\sqrt{2\pi}\,n} + O\left(\tfrac1{n^2}\right).
\]
\end{proof}

\begin{proof}[Proof of Theorem~\ref{Thm:Normal_appr}]
\label{Page:proof-of-Thm:Normal_appr}
Inequality~\eqref{Eq:Main_normal} results from  \eqref{Eq:Main_normal(noniid)} in Theorem~\ref{Thm:Main(noniid)},
already proved above, when specialized to the i.i.d.~case.
Alternatively, we may first specialize Theorem~\ref{Rem:Main_thm_rewritten} to the i.i.d.~case
and then apply the triangle inequality similarly to~\eqref{Eq:Final_display} below. 

Let now $\rho\in[1,\infty[$ and $f(x)=x^3/6$ for $x\in\R$. Then we have 
\[
 \sqrt{k}\left|\widetilde{P_\rho^{\ast k}}f-\mathrm{N}f \right|
   &=&  \sqrt{k}\left|\widetilde{P_\rho^{\ast k}}f  \right| 
   \,\ =\,\ \frac1k\left| P_\rho^{\ast k} f  \right| 
   \,\ =\,\ \frac{B(\rho)}{6} \quad\text{ for }k\in\N
\]
by using in the last step above \eqref{Eq:Cumulants_additive}  with $\ell=3$,
as we did in Step~1 of the proof of Theorem~\ref{Thm:Main}, and hence we get 
\[
 \frac{B(\rho)}{6} &\le& \varliminf_{n\rightarrow\infty} \sqrt{n}\zeta_3\left(\widetilde{P_\rho^{\ast n}} ,\mathrm{N} \right)
   \,\ \le \,\ \varlimsup_{n\rightarrow\infty} \sqrt{n}\zeta_3\left(\widetilde{P_\rho^{\ast n}} ,\mathrm{N} \right)
 \,\ \le \,\ \frac{B(\rho)}{6}
\]
using in the last step~\eqref{Eq:Main_normal} with $P=P_\rho$ and $\epsilon_n=O(n^{-1})$. 
This proves~\eqref{Eq:B-E_normal_opt_for_n_large}. 

Let finally $n\in\N$. For $P\in\cP_3$ using the triangle inequality for $\zeta_3$ in the first step below and  the i.i.d.~case of Theorem~\ref{Rem:Main_thm_rewritten} in the second we then have
\la                                                                \label{Eq:Final_display}
  \left|\zeta_3\left(\widetilde{P^{\ast n}},\mathrm{N}\right) -\eps_n\right|
  &\le&  \zeta_3\left(\widetilde{P^{\ast n}}, \widetilde{\mathrm{B}_{n,\frac12}}\right)
   \,\ \le  \,\  \tfrac16{B(\rho(P))},
\al
and~\eqref{Eq::B-E_normal_opt_for_rho_small} follows using  $\lim_{\rho\rightarrow1}B(\rho)=0$.
\end{proof}

\section*{Acknowledgements}
We thank J\"urgen M\"uller for showing us the limit result in
\cite{Hardy1949} and \cite{Wirths2014} used in the proof of Lemma~\ref{Lem:zeta(1/2)}, and Bero Roos for helpful remarks on an earlier version of the present paper.

\end{document}